\documentclass[11pt,a4paper]{article} 

\usepackage[latin1]{inputenc}
\usepackage{amsmath,amssymb,amsthm,srcltx}
\usepackage{mathtools}
\usepackage[dvips]{color}
\usepackage[margin]{fixme}
\usepackage{enumerate}
\usepackage{stmaryrd}
\usepackage[colorinlistoftodos]{todonotes}
\usepackage{bbm}
\usepackage{dsfont}
\numberwithin{equation}{section}
\usepackage{fullpage}

\newcommand{\DD}{\mathbb{D}}
\newcommand{\EE}{\mathbb{E}}
\newcommand{\FF}{\mathbb{F}}
\newcommand{\GG}{\mathbb{G}}
\newcommand{\HH}{\mathbb{H}}
\newcommand{\II}{\mathbb{I}}

\newcommand{\LL}{\mathbb{L}}

\newcommand{\NN}{\mathbb{N}}
\newcommand{\PP}{\mathbb{P}}
\newcommand{\QQ}{\mathbb{Q}}
\newcommand{\RR}{\mathbb{R}}
\newcommand{\SSS}{\mathbb{S}}
\newcommand{\TT}{\mathbb{T}}

\newcommand{\YY}{\mathbb{Y}}

\newcommand{\cB}{\mathcal{B}}
\newcommand{\cC}{\mathcal{C}}
\newcommand{\cD}{\mathcal{D}}
\newcommand{\cE}{\mathcal{E}}
\newcommand{\cF}{\mathcal{F}}
\newcommand{\cG}{\mathcal{G}}
\newcommand{\cH}{\mathcal{H}}
\newcommand{\cI}{\mathcal{I}}

\newcommand{\cL}{\mathcal{L}}
\newcommand{\cM}{\mathcal{M}}
\newcommand{\cN}{\mathcal{N}}
\newcommand{\cP}{\mathcal{P}}
\newcommand{\cQ}{\mathcal{Q}}
\newcommand{\cS}{\mathcal{S}}
\newcommand{\cT}{\mathcal{T}}

\newcommand{\cY}{\mathcal{Y}}
\newcommand{\cZ}{\mathcal{Z}}

\def\no{\noindent}

\def\ms{\medskip}

\def\q{\quad}

\def\dbD{\mathbb{D}}
\def\dbE{\mathbb{E}}
\def\dbF{\mathbb{F}}

\def\dbH{\mathbb{H}}

\def\dbN{\mathbb{N}}
\def\dbP{\mathbb{P}}
\def\dbR{\mathbb{R}}

\def\dbS{\mathbb{S}}

\def\dbQ{\mathbb{Q}}

\def\b{\beta}
\def\d{\delta}
\def\e{\varepsilon}

\def\om{\omega}

\newcommand{\ba}{\begin{array}} 
\newcommand{\ea}{\end{array}}
\newcommand{\be}{\begin{equation}}
\newcommand{\ee}{\end{equation}}
\newcommand{\bea}{\begin{eqnarray}}
\newcommand{\eea}{\end{eqnarray}}
\newcommand{\beaa}{\begin{eqnarray*}}
\newcommand{\eeaa}{\end{eqnarray*}}

\newcommand{\Ind}{\mathbf{1}}

\DeclareMathOperator\sgn{\mathrm{sgn}}

\DeclareMathOperator*{\esssup}{ess\,sup}
\DeclareMathOperator*{\essinf}{ess\,inf}

\newtheorem{theorem}{Theorem}[section] 
\newtheorem{assumption}[theorem]{Assumption}

\newtheorem{corollary}[theorem]{Corollary}

\newtheorem{definition}[theorem]{Definition}
\newtheorem{lemma}[theorem]{Lemma}
\newtheorem{proposition}[theorem]{Proposition}

\theoremstyle{definition}

\newtheorem{remark}[theorem]{Remark}

\begin{document}

\title{Nonlinear predictable representation and $\LL^1$-solutions of backward SDEs and second-order backward SDEs}

\author{Zhenjie REN\footnote{CEREMADE, Universit\'e Paris Dauphine, F-75775 Paris Cedex 16, France, {\tt ren@ceremade.dauphine.fr}. }
        \and  Nizar TOUZI\footnote{CMAP, \'Ecole Polytechnique, F-91128 Palaiseau Cedex, France, {\tt nizar.touzi@polytechnique.edu}.}
        \and Junjian YANG\footnote{FAM, Fakult\"at f\"ur Mathematik und Geoinformation, Vienna University of Technology, A-1040 Vienna, Austria, {\tt junjian.yang@tuwien.ac.at}.}
       }

\date{\today}

\maketitle

\begin{abstract} 
The theory of backward SDEs extends the predictable representation property of Brownian motion to the nonlinear framework, thus providing a path-dependent analog of fully nonlinear parabolic PDEs. 
In this paper, we consider backward SDEs, their reflected version, and their second-order extension, in the context where the final data and the generator satisfy $\LL^1$-integrability condition. 
Our main objective is to provide the corresponding existence and uniqueness results for general Lipschitz generators. 
The uniqueness holds in the so-called Doob class of processes, simultaneously under an appropriate class of measures. 
We emphasize that the previous literature only deals with backward SDEs, and requires either that the generator is separable in $(y,z)$, see Peng \cite{Peng97}, or strictly sublinear in the gradient variable $z$, see Briand, Delyon, Hu, Pardoux \& Stoica  \cite{BDHPS03}, or that the final data satisfies an $L\ln{L}$-integrability condition, see Hu \& Tang \cite{HT18}. 
We bypass these conditions by defining $\LL^1$-integrability under the nonlinear expectation operator induced by the previously mentioned class of measures.

\vspace{3mm}

\begin{center}
	\textbf{Résumé}
\end{center}

La th\'eorie des \'equations diff\'erentielles stochastiques r\'etrogrades \'etend la propri\'et\'e de repr\'esentation pr\'evisible du mouvement brownien au cadre non lin\'eaire, offrant ainsi un analogue non-markovien aux \'equations aux d\'eriv\'ees partielles paraboliques compl\`etement non lin\'eaires. 
Dans ce papier, nous consid\'erons les EDS r\'etrogrades, leur version r\'efl\'echies et son extension au second ordre, dans le contexte d'une donn\'ee terminale et d'un g\'en\'erateur $\mathbb{L}^1$-int\'egrables. 
Notre objectif est d'\'etablir un r\'esultat d'existence et d'unicit\'e pour un g\'en\'erateur Lipschitzien. Nous montrons que l'unicit\'e a lieu dans la classe des processus de Doob, simultan\'ement sous une classe appropri\'ees de mesures de probabilit\'e sur l'espace des trajectoires. 
Notons que ce r\'esultat est nouveau, m\^eme dans le cas particulier des EDS r\'etrogrades, o\`u la litt\'erature pr\'ec\'edente \'etablit l'unicit\'e pour des g\'en\'erateurs, soit s\'eparables en $(y,z)$ (Peng \cite{Peng97}), soit strictement sous-lin\'eaires en la variable de gradient $z$, (Briand, Delyon, Hu, Pardoux \& Stoica \cite{BDHPS03}), ou alors sous une condition d'int\'egrabilit\'e $L\ln{L}$ (Hu \& Tang \cite{HT18}). 
Nous \'evitons de recourir \`a de telles conditions en introduisant l'int\'egrabilit\'e $\mathbb{L}^1$ sous l'op\'erateur d'esp\'erance non lin\'eaire induit par la classe ci-dessus de mesures de probabilit\'e.

\end{abstract}

\vspace{-0.2cm}\newline


\section{Introduction}

Backward stochastic differential equations extend the martingale representation theorem to the nonlinear setting. It is well-known that the martingale representation theorem is the path-dependent counterpart of the heat equation. 
Similarly, it has been proved in the seminal paper of Pardoux and Peng \cite{PP90} that backward SDEs provide a path-dependent substitute to semilinear PDEs. Finally, the path-dependent counterpart of parabolic fully nonlinear parabolic PDEs was obtained by Soner, Touzi \& Zhang \cite{STZ12} and later by Hu, Ji, Peng \& Song \cite{HJPS14a, HJPS14b}. 
The standard case of a Lipschitz nonlinearity (or generator), has been studied extensively in the literature, the solution is defined on an appropriate $\LL^p$-space for some $p>1$, and wellposedness is guaranteed whenever the final data and the generator are $\LL^p$-integrable. 

In this paper, our interest is on the limiting $\LL^1$-case. 
It is well-known that the martingale representation, which is first proved for square integrable random variables, holds also in $\LL^1$ by a density argument. 
This is closely related to the connexion with the conditional expectation operator. 

The first attempt for an $\LL^1$-theory of backward SDEs is by Peng \cite{Peng97} in the context of a separable nonlinearity $f_1(t,y)+f_2(t,z)$, Lipschitz in $(y,z)$, with $f_1(t,0)=0$, $f_2(t,0)\geq 0$, and final data $\xi\geq 0$. 
The wellposedness result of this paper is specific to the scalar case, and follows the lines of the extension of the expectation operator to $\LL^1$.

Afterwards, Briand, Delyon, Hu, Pardoux \& Stoica \cite{BDHPS03} consider the case of multi-dimensional backward SDEs, and obtain a wellposedness result in $\LL^1$ by using a truncation technique leading to a Cauchy sequence. 
This approach is extended by Rozkosz \& S{\l}omi{\'n}ski \cite{RS12} and Klimsiak \cite{kli12} to the context of reflected backward SDEs.
However, the main result of these papers requires the nonlinearity to be strictly sublinear in the gradient variable. In particular, this does not cover the linear case, whose unique solution is immediately obtained by a change of measure. More generally, the last restriction excludes the nonlinearities generated by stochastic control problems (with uncontrolled diffusion), which is a substantial field of application of backward SDEs, see El Karoui, Peng \& Quenez \cite{EKPQ97} and Cvitani\'c, Possama\"{\i} \& Touzi \cite{CPT18}. 

We finally refer to the recent work series by Buckdahn, Fan, Hu and Tang \cite{HT18,BHT18,FH19} who provide an $L\exp(\mu\sqrt{2\log(1+L)})$-integrability condition which guarantees the wellposedness in $\LL^1$ of the backward SDE for a Lipschitz nonlinearity.

In this paper, we consider an alternative integrability class for the solution of the backward SDE by requiring $\LL^1$-integrability under a nonlinear expectation induced by an appropriate family of probability measures. In the context of a Lipschitz nonlinearity, the first main result of this paper provides wellposedness of the backward SDE for a final condition and a nonlinearity satisfying a uniform integrability type of condition under the same nonlinear expectation. This result is obtained by appropriately adapting the arguments of \cite{BDHPS03}. Although all of our results are stated in the one-dimensional framework, we emphasize that the arguments used for the last wellposedness results are unchanged in the multi-dimensional context. 

We also provide a similar wellposedness result for (scalar) reflected backward SDEs, under the same conditions as for the corresponding backward SDE, with an obstacle process whose positive value satisfies the same type of uniform integrability under nonlinear expectation. This improves the existence and uniqueness results of \cite{RS12, kli12}.

Our third main result is the wellposedness of second order backward SDEs in $\LL^1$. Here again, the $\LL^1$-integrability is in the sense of a nonlinear expectation induced by a family of measures. In the present setting, unlike the case of backward SDEs and their reflected version, the family of measures is non-dominated as in Soner, Touzi \& Zhang \cite{STZ12} and Possama\"{\i}, Tan and Zhou \cite{PTZ18}.
 
 \vspace{3mm}
 
The paper is organized as follows. Section \ref{sect:preliminaries} introduces the notations used throughout the paper. Our main results are contained in Section \ref{sect:mainresults}, with proofs postponed in the rest of the paper. Section \ref{sect:RBSDE} contains the proofs related to (reflected) backward SDEs, and Sections \ref{sect:2BSDEuniqueness} and \ref{sect:2BSDEexistence} focus on the uniqueness and the existence, respectively, for the second-order backward SDEs.

\section{Preliminaries}
\label{sect:preliminaries}

\subsection{Canonical space}

For a given fixed maturity $T>0$ and $d\in\NN$, we denote by 
  $$ \Omega := \left\{\omega\in\cC\big([0,T];\RR^d\big)\,:\, \omega_0=\bf 0 \right\} $$
 the canonical space equipped with the norm of uniform convergence $\|\omega\|_{\infty}:=\sup_{0\leq t\leq T}|\omega_t|$ and by $B$ the canonical process.
Let $\cM_1$ be the collection of all probability measures on $(\Omega,\cF)$, equipped with the topology of weak convergence. 
Let $\FF:=\{\cF_t\}_{0\leq t\leq T}$ be the raw filtration generated by the canonical process $B$, and $\FF^{+}:=\big\{\cF^{+}_t\big\}_{0\leq t\leq T}$ the right limit of $\FF$. 
For each $\PP\in\cM_1$, we denote by $\FF^{+,\PP}$ the augmented filtration of $\FF^+$ under $\PP$.
The filtration $\FF^{+,\PP}$ is the coarsest filtration satisfying the usual conditions. 
Let $\dbP_0$ be the Wiener measure. 
Moreover, for $\cP\subseteq\cM_1$, we introduce the following filtration $\FF^U:=(\cF^U_t)_{0\leq t\leq T}$, $\FF^\cP:=(\cF_t^\cP)_{0\leq t\leq T}$, and $\FF^{+,\cP}:=\big(\cF_t^{+,\cP}\big)_{0\leq t\leq T}$, 
  defined as follows
  $$ \cF^U_t:= \bigcap_{\PP\in\cM_1}\cF_t^\PP, \,\quad  \cF^\cP_t:= \bigcap_{\PP\in\cP}\cF_t^\PP, \quad \cF_t^{+,\cP}:=\cF^{\cP}_{t+}, \,\, t\in[0,T), \, \mbox{ and }\,\, \cF_T^{+,\cP}:=\cF_T^\cP, $$
  where $\cF_t^\dbP$ denotes the $\dbP$-completion of $\cF_t$. 
The filtration $\FF^U$ is called the universally completed filtration.
For any family $\cP\subseteq\cM_1$, we say that a property holds $\cP$-quasi-surely, abbreviated as $\cP$-q.s., if it holds $\dbP$-a.s.~for all $\dbP\in\cP$.
Finally, for $0\le s\le t\le T$, we denote by $\cT_{s,t}$ the collection of all $[s,t]$-valued $\dbF$-stopping times.

\subsection{Local martingale measures}

We say a probability measure $\dbP$ is a local martingale measure if the canonical process $B$ is a local martingale under $\dbP$. 
By \cite{Kar95}, we can pathwisely define a $d\times d$-matrix-valued $\FF$-progressively measurable process $\langle B\rangle$ which coincides with the cross-variation of $B$ under each local martingale measure $\PP$, 
  and its density by 
\beaa
 \widehat{a}_t:=\limsup_{\varepsilon\searrow 0}\frac{\langle B\rangle_t - \langle B\rangle_{t-\varepsilon}}{\varepsilon},
\eeaa
Let $\overline{\cP}_W$ denote the set of all local martingale measures $\dbP$ 
  such that $\langle B\rangle_t$ is absolutely continuous in $t$ and $\widehat{a}$ takes values in $\SSS_d^{>0}$ (the set of $d\times d$ real valued symmetric positive-definite matrices) $\dbP$-a.s.
We note that, for different $\dbP_1$, $\dbP_2\in\overline{\cP}_W$, in general $\dbP_1$ and $\dbP_2$ are mutually singular. See \cite[Example 2.1]{STZ12} for an example. 
For any $\dbP\in\overline{\cP}_W$, it follows from the L\'evy characterization that the process defined by the following It\^o's integral 
   $$ W^\dbP_t:=\int_0^t\widehat{a}_s^{-1/2}dB_s, \quad t\in[0,T],\quad \dbP\mbox{-a.s.} $$
   is a Brownian motion under $\dbP$. 
We define a subset $\overline{\cP}_S\subseteq\overline{\cP}_W$ consisting of all probability measures 
  $$ \dbP^\alpha:=\dbP_0\circ(X^\alpha)^{-1}, \quad \mbox{where}~~~X^\alpha_t:=\int_0^t\alpha_s^{1/2}dB_s,~~t\in[0,T],~~\dbP_0\mbox{-}a.s. $$
  for some $\dbF$-progressively measurable process $\alpha$ taking values in $\SSS_d^{>0}$ with $\int_0^T|\alpha_s|dt<\infty$, $\dbP_0$-a.s. 
It follows from \cite{STZ11a} that 
  $$ \overline{\cP}_S = \left\{\dbP\in\overline{\cP}_W: \overline{\dbF^{W^\dbP}}^\dbP=\overline{\dbF}^\dbP\right\}, $$
  where $\overline{\dbF}^\dbP$ ($\overline{\dbF^{W^\dbP}}^\dbP$, respectively) denotes the $\dbP$-augmentation of the right-limit filtration generated by $B$ (by $W^\dbP$, respectively). 
Moreover, every $\dbP\in\overline{\cP}_S$ satisfies the Blumenthal zero-one law and the martingale representation property. 

We also recall that, for $\dbP\in\overline{\cP}_S$, it follows from the Blumenthal zero-one law that 
 $\dbE^\dbP[\xi|\cF_t]=\dbE^\dbP[\xi|\cF_t^+]$, $\dbP$-a.s., for any $t\in[0,T]$ and $\dbP$-integrable $\xi$. In particular, any $\cF_t^+$-measurable random variable has an $\cF_t$-measurable $\dbP$-modification. 
We refer the reader to \cite[Section 6]{STZ11a} for more details about $\overline{\cP}_S$ and $\overline{\cP}_W$. 

\subsection{Spaces and norms}\label{sec:space}

\textbf{(i)} \textit{One-measure integrability classes}: 
For a probability measure $\PP$ and $p>0$, we denote: 
\begin{itemize}
 \item $\LL^p(\PP)$ is the space of $\RR$-valued and $\cF^{+,\PP}_T$-measurable random variables $\xi$, such that 
                $$ \|\xi\|_{\LL^p(\PP)} := \EE^{\PP}\left[|\xi|^p\right]^{1\wedge\frac{1}{p}} <\infty. $$
 \item $\SSS^p(\PP)$ is the space of $\RR$-valued, $\FF^{+,\PP}$-adapted processes $Y$ with c\`adl\`ag paths, such that 
                $$ \|Y\|_{\SSS^p(\PP)} := \EE^{\PP}\left[\sup_{0\leq t\leq T}|Y_t|^p\right]^{1\wedge\frac{1}{p}} <\infty. $$
 \item $\HH^p(\PP)$ is the space of $\RR^d$-valued, $\FF^{+,\PP}$-progressively measurable processes $Z$ such that 
                $$ \|Z\|_{\HH^p(\PP)}:= \EE^{\PP}\left[\left(\int_0^T\left|\widehat\sigma_s^{\intercal}Z_s\right|^2ds\right)^{\frac{p}{2}}\right]^{1\wedge\frac{1}{p}} <\infty. $$
 \item $\II^p(\PP)$ is the set of $\RR$-valued, $\FF^{+,\PP}$-predictable processes $K$ of bounded variation with c\`adl\`ag nondecreasing paths, such that
           $$ \|K\|_{\II^p(\PP)} := \EE^{\PP}\left[K_T^p\right]^{1\wedge\frac{1}{p}} <\infty. $$
 \end{itemize}

\no The spaces above are Banach spaces with corresponding norms for $p\geq 1$ and complete metric spaces with corresponding metrics if $p\in(0,1)$. 
A process $Y$ is said to be of class $\DD(\PP)$ if the family $\{Y_\tau, \,\tau\in\cT_{0,T}\}$ is uniformly integrable under $\PP$. 
We define the norm 
 \begin{equation*}
   \|Y\|_{\DD(\PP)}:=\sup_{\tau\in\cT_{0,T}}\EE^\PP[|Y_\tau|]. 
 \end{equation*}
The space of progressive measurable c\`adl\`ag processes which belong to class $\DD(\PP)$ is complete under this norm. 
See Theorem \cite[VI Theorem 22, Page 83]{DM82}.

\vspace{4mm}

\noindent \textbf{(ii)} \textit{Integrability classes under dominated nonlinear expectation}: 
Let us denote by $\widehat\sigma$ the measurable square root of $\widehat a_t$, i.e., $\widehat\sigma_t:=\widehat a_t^{1/2}$. We recall that, for each $\dbP\in\overline{\cP}_S$, 
  $ W^\dbP_t := \int_0^t\widehat{\sigma}_s^{-1}dB_s $
  is a $\dbP$-Brownian motion. 
Denote by $\cQ_{L}(\PP)$ the set of all probability measures $\QQ^{\lambda}$ such that
   $$ \frac{d\QQ^\lambda}{d\PP}\bigg|_{\cF_t} = G^\lambda_t := \exp\left\{\int_0^t\lambda_s\cdot dW^\dbP_s - \frac{1}{2}\int_0^t|\lambda_s|^2ds\right\}, \quad t\in[0,T],  $$
   for some $\FF^{+,\PP}$-progressively measurable process $(\lambda_t)_{0\leq t\leq T}$ bounded uniformly by $L$.
It is straightforward to check that the set $\cQ_L(\PP)$ is stable under concatenation, i.e., 
 for $\QQ_1$, $\QQ_2\in\cQ_L(\PP)$, $\tau\in\cT_{0,T}$, we have $\QQ_1\otimes_\tau\QQ_2\in\cQ_L(\PP)$, where 
   \begin{align*}
     \QQ_1\otimes_\tau\QQ_2(A):= \EE^{\QQ_1}\Big[\EE^{\QQ_2}\big[\Ind_A|\cF^{+,\dbP}_\tau\big]\Big], \quad A\in\cF^{+,\dbP}_T.
   \end{align*}
It is clear from Girsanov's Theorem that under a measure $\QQ^\lambda\in\cQ_{L}(\PP)$, the process $W_t^{\lambda}:=W_t^\dbP-\int_0^{t}\lambda_sds$ is a Brownian motion under $\QQ^\lambda$.
Thus, $B_t^\lambda:=B_t-\int_0^t\widehat{\sigma}_t\lambda_tdt$ is a $\QQ^\lambda$-martingale. 
Given a $\PP\in\overline{\cP}_S$, we denote 
   $$ \cE^{\PP}[\xi]:= \sup_{\QQ\in\cQ_L(\PP)}\EE^\QQ[\xi], $$
  and introduce the space $\cL^p(\PP)=\bigcap_{\QQ\in\cQ_L(\PP)}\LL^p(\QQ)$ of $\cF_T^{+,\dbP}$-measurable random variables $\xi$ such that 
    $$ \|\xi\|_{\cL^p(\PP)} := \cE^{\PP}\left[|\xi|^p\right]^{1\wedge\frac{1}{p}} <\infty.  $$
We define similarly the subspaces $\cS^p(\PP)$, $\cH^p(\PP)$ and the subsets $\cI^p(\PP)$.

A  process $Y$ belongs to $\cD(\PP)$ if $Y$ is progressive measurable and c\`adl\`ag, and  the family $\{Y_\tau, \,\tau\in\cT_{0,T}\}$ is uniformly integrable under $\cQ_L(\PP)$, i.e., 
 $ \lim_{N\to\infty}\sup_{\tau\in\cT_{0,T}}\cE^{\PP}\left[|Y_\tau|\Ind_{\{|Y_\tau|\geq N\}}\right] = 0. $
We define the norm 
\beaa
\|Y\|_{\cD(\PP)}:=\sup_{\tau\in\cT_{0,T}}\cE^\PP[|Y_\tau|]. 
\eeaa
Note that $ \|Y\|_{\cD(\PP)}<\infty$ does not imply $Y\in \cD(\PP)$. However, the space $\cD(\PP)$ is complete under this norm. See Theorem \ref{AppThm1}.

\vspace{4mm}

\noindent \textbf{(iii)} \textit{Integrability classes under non-dominated nonlinear expectation}: 
 Let $\cP\subseteq\overline{\cP}_S$ be a subset of probability measures, and denote
  $$ \cE^\cP[\xi]:=\sup_{\PP\in\cP}\cE^\PP[\xi].  $$ 
%
 Let $\GG:=\{\cG_t\}_{0\leq t\leq T}$ be a filtration with $\cG_t\supseteq\cF_t$ for all $0\leq t\leq T$. 
 We define the subspace $\cL^p(\cP, \cG_T)$ as the collection of all $\cG_T$-measurable $\mathbb{R}$-valued random variables $\xi$, such that 
     $$ \|\xi\|_{\cL^p(\cP)} := \cE^{\cP}\left[|\xi|^p\right]^{1\wedge\frac{1}{p}} <\infty. $$
 We define similarly the subspaces $\cS^p(\cP, \GG)$ and $\cH^p(\cP, \GG)$ by replacing $\mathbb{F}^{\PP}$ with $\mathbb{G}$.
 Similarly, we denote by $\cD(\cP,\GG)$ the space of $\RR$-valued, $\GG$-adapted processes $Y$ with c\`adl\`ag paths, such that 
                $$ \lim_{N\to\infty}\sup_{\tau\in\cT_{0,T}}\cE^{\cP}\big[|Y_\tau|\Ind_{\{|Y_\tau|\geq N\}}\big] = 0. $$

\subsection{Shifted space}

 Here, we introduce the following notations:
 \begin{itemize}
  \item For $t\in[0,T]$ and $\omega$, $\omega'\in\Omega$, define the concatenation path $\omega\otimes_t\omega'\in\Omega$ by 
          $$ (\omega\otimes_t\omega')_s:=\omega_s\Ind_{[0,t)}(s) + (\omega_t+\omega'_{s-t})\Ind_{[t,T]}(s), \quad s\in[0,T]. $$
  \item For $(t,\omega)\in\Omega\times[0,T]$, we define the shifted random variable 
          $$ \xi^{t,\omega}(\omega'):=\xi(\omega\otimes_t\omega'), \quad \forall \omega'\in\Omega. $$
        If $\xi$ is $\cF_{t+s}$-measurable, then $\xi^{t,\omega}$ is $\cF_s$-measurable. 
  \item For an $\FF$-stopping time $\tau$ and $(t,\omega)\in\llbracket 0,\tau\llbracket$, define $\theta:=\tau^{t,\omega}-t.$
        One may show that $\theta$ is a $\FF$-stopping time.
  \item Define the shifted process 
          $$ Y^{t,\omega}_s(\omega'):= Y_{t+s}(\omega\otimes_t\omega'). $$
        In particular, for the canonical process $X$, we have 
          $$ X^{t,\omega}_s(\omega')= X_{t+s}(\omega\otimes_t\omega') = (\omega\otimes_t\omega')_{t+s} = \omega_t + \omega'_s, \quad s\in[0,T-t]. $$
  \item For $\FF$-stopping time $\tau$, we define 
          $$ \omega\otimes_{\tau}\omega':=\omega\otimes_{\tau(\omega)}\omega', \quad \xi^{\tau,\omega}:=\xi^{\tau(\omega),\omega}, \quad Y^{\tau,\omega}:=Y^{\tau(\omega),\omega}. $$ 
 \end{itemize}

 For every probability measure $\PP$ on $\Omega$ and $\FF$-stopping time $\tau$, there exists a family of regular conditional probability distribution (for short r.c.p.d.) $(\PP^\tau_\omega)_{\omega\in\Omega}$ satisfying 
 \begin{itemize}
  \item For each $\omega\in\Omega$, $\PP^\tau_\omega$ is a probability measure on $(\Omega,\cF_T)$.
  \item For each $A\in\cF_T$, the mapping $\omega\mapsto\PP^\tau_\omega(A)$ is $\cF_\tau$-measurable. 
  \item For $\PP$-a.e.~$\omega\in\Omega$, $\PP^\tau_\omega$ is the conditional probability measure of $\PP$ on $\cF_\tau$, i.e., for every integrable $\cF_T$ random variable $\xi$ we have 
          $$ \EE^\PP[\xi|\cF_\tau](\omega)=\EE^{\PP^\tau_\omega}[\xi], \quad \mbox{ for  }\PP\mbox{-a.e. } \omega\in\Omega. $$
  \item For every $\omega\in\Omega$, $\PP^\tau_\omega(\Omega^\omega_\tau)=1$, where 
          $$ \Omega^\omega_\tau := \left\{\overline\omega\in\Omega\,|\,\overline{\omega}_s = \omega_s, \, 0\leq s\leq\tau(\omega)\right\}. $$
 \end{itemize}
 
The r.c.p.d.~$\PP^{\tau}_\omega$ induces naturally a probability measure $\PP^{\tau,\omega}$ on $(\Omega,\cF_T)$ defined by 
  $$ \PP^{\tau,\omega}(A):=\PP_\omega^\tau(\omega\otimes_\tau A), \quad A\in\cF_T, \quad \mbox{where }\, \omega\otimes_\tau A:=\{\omega\otimes_\tau\omega'\,:\,\omega'\in A\}. $$
It is clear that for every $\cF_T$-measurable random variable $\xi$, we have
  $$ \EE^{\PP^{\tau}_\omega}[\xi] = \EE^{\PP^{\tau,\omega}}[\xi^{\tau,\omega}]. $$

\section{Main results}
\label{sect:mainresults}

Throughout this paper, we fix a finite time horizon $0<T<\infty$. Let $\xi$ be an $\cF_T$-measurable random variable, and $F:[0,T]\times\Omega\times\RR\times\RR^d\times\SSS^{d}\to\RR$, a Prog$\otimes\cB(\RR)\otimes\cB(\RR^d)\otimes\cB\big(\SSS^{d}\big)$-measurable map,\footnote{We denote by Prog the $\sigma$-algebra generated by progressively measurable processes.} called generator, and denote
   $$ 
   f_t(\omega,y,z)
   := 
   F_t\big(\omega,y,z,\widehat{\sigma}_t(\omega)\big), 
   \quad (t,\omega,y,z)\in[0,T]\times\Omega\times\RR\times\RR^d.
   $$
By freezing the pair $(y,z)$ to $0$, we set $f^0_t=f_t(0,0)$. 

\begin{assumption}\label{assum:BSDE.Lipschitz}
  The coefficient $F$ is uniformly Lipschitz in $(y, z)$ in the following sense: there exist constants $L_y$, $L_z\geq 0$, such that for all $(y_1, z_1)$, $(y_2, z_2)\in \mathbb{R}\times\RR^d$ and $\sigma\in\SSS^{d}$,
          $$ \big|F_s(y_1, z_1,\sigma)-F_s(y_2, z_2,\sigma)\big|\leq L_y|y_1-y_2|+L_z\big|\sigma^{\intercal}(z_1-z_2)\big|, \quad ds\otimes d\mathbb{P}\mbox{-a.s.} $$
\end{assumption}

\begin{remark} \label{rem:monotonicity}
Without loss of generality, we may assume that $F$ is nonincreasing in $y$. Indeed, we may always reduce to this context by using the standard change of variable $(\widetilde Y_t,\widetilde Z_t):=e^{at}(Y_t,Z_t)$ for sufficiently large $a$. 
\end{remark}
 
\subsection{$\LL^1$-solution of backward SDE} \label{subsection:BSDE}

For a probability measure $\PP\in\overline{\cP}_S$, consider the following backward stochastic differential equation (BSDE): 
   \begin{equation} \label{eq:BSDE}
     Y_t = \xi+\int_{t}^T f_s(Y_s,Z_s)ds - \int_{t}^TZ_s\cdot dB_s, \quad t\in[0,T],\quad \PP\mbox{-a.s.}
   \end{equation}
Here, $Y$ is a c\`adl\`ag adapted $\RR$-valued process and $Z$ is a progressively measurable $\RR^d$-valued process. 
Recall that $dB_s=\widehat{\sigma}_sdW_s^\dbP$, $\PP$-a.s.~for some $\PP$-Brownian motion $W^\dbP$.

We shall use the Lipschitz constant $L_z$ of Assumption \ref{assum:BSDE.Lipschitz} as the bound of the coefficients of the Girsanov  transformations introduced in Section \ref{sec:space} (ii). In particular, we denote $$ \cE^\PP[\xi]:= \sup_{\QQ\in\cQ_{L_z}(\PP)}\EE^\QQ[\xi]. $$

\begin{assumption} \label{assum:BSDExi}
	Suppose that 
 $$   \lim_{n\to\infty}\cE^{\PP}\left[|\xi|\Ind_{\{|\xi|\geq n\}} + \int_0^T\big|f_s^0\big|\Ind_{\{|f^0_s|\geq n\}}ds\right] = 0, $$
  where $\cE^{\PP}$ is defined above with some Lipschitz constant for $F$ in $z$. 
\end{assumption}

\begin{remark}
 Note that for the linear expectation, the integrability of a random variable $\xi$ is equivalent to $\lim_{n\to\infty}\dbE\big[|\xi|\Ind_{\{\xi\geq n\}}\big]=0$. 
 However, this equivalence does not hold true any more for nonlinear expectations, see Proposition \ref{AppProp1}, where we need the condition $(b)$ in order to get the equivalence. 
 Such an assumption is usual in the nonlinear expectation framework.
\end{remark}
 
\begin{remark}[Comparison with Hu \& Tang \cite{HT18} and Buckdahn, Hu \& Tang \cite{BHT18}]
 Under the previous Lipschitz condition, and for $\mu>L_z\sqrt{T}$, let $\psi(x):= x\exp\big(\mu\sqrt{2\log(1+x)}\big)$ for $x\in[0,\infty)$, and assume that 
  $\psi\Big(|\xi|+\int_0^T|f^0_s|ds\Big)\in \LL^1(\dbP)$. 
 Then, \cite{HT18} and \cite{BHT18} proved that the BSDE \eqref{eq:BSDE} has a unique solution with $\{\psi(|Y_t|),0\leq t\leq T\}\in\dbD(\dbP)$. 
 They show in \cite[Remark 1.2]{HT18} that the $\psi$-integrability is stronger than $\LL^1$, weaker than $\LL^p$ for any $p>1$. 
 We now show our integrability assumption is weaker than their, i.e., $\psi(|\xi|)\in\LL^1(\dbP)$ implies $\lim_{n\to\infty}\cE^\dbP\big[|\xi|\Ind_{\{|\xi|\geq n\}}\big]=0$. 
%
%
 
 To see this, let $\dbQ^\lambda\in\cQ_L(\dbP)$ be arbitrary. By the inequality,  
  $e^xy\leq e^{\frac{x^2}{2\mu^2}}+e^{2\mu^2}\psi(y)$ with $x,y\geq 0$, see \cite[Lemma 2.4]{HT18},
   we obtain 
    \begin{align*}
      \dbE^{\dbQ^\lambda}\left[|\xi|\Ind_{\{|\xi|\geq N\}}\right] 
         &= \dbE\left[|\xi|\Ind_{\{|\xi|\geq N\}}e^{\int_0^T\lambda_s\cdot dW_s -\frac{1}{2}\int_0^T|\lambda_s|^2ds}\right] \leq \dbE\left[|\xi|\Ind_{\{|\xi|\geq N\}}e^{\int_0^T\lambda_s\cdot dW_s}\right] \\
         &\leq \dbE\left[\Ind_{\{|\xi|\geq N\}}e^{\frac{1}{2\mu^2}(\int_0^T\lambda_s\cdot dW_s)^2}\right] + e^{2\mu^2}\dbE\left[\Ind_{\{|\xi|\geq N\}}\psi(|\xi|)\right].
    \end{align*}
 The second term is independent of $\dbQ^\lambda$ and converges to $0$ as $N\to\infty$, by the $\psi$-integrabiliy of $\xi$. 
 For the first term, following the proof of \cite[Lemma 2.6, Theorem 2.7]{HT18}, we obtain that
   \begin{align*}
     \dbE\left[\Ind_{\{|\xi|\geq N\}}e^{\frac{1}{2\mu^2}(\int_0^T\lambda_s\cdot dW_s)^2}\right]
       & = \frac{1}{\sqrt{2\pi}}\int_{-\infty}^\infty \dbE\left[\Ind_{\{|\xi|\geq N\}}e^{\frac{y}{\mu}\int_0^T\lambda_s\cdot dW_s}\right]e^{-\frac{y^2}{2}}dy \\
       &\leq \frac{1}{\sqrt{2\pi}}\int_{-\infty}^\infty \dbE\left[\Ind_{\{|\xi|\geq N\}}^q\right]^{\frac{1}{q}}\dbE\left[e^{p\frac{y}{\mu}\int_0^T\lambda_s\cdot dW_s}\right]^{\frac{1}{p}}e^{-\frac{y^2}{2}}dy \\
       &\leq \dbE\left[\Ind_{\{|\xi|\geq N\}}\right]^{\frac{1}{q}} \frac{1}{\sqrt{2\pi}}\int_{-\infty}^\infty e^{\frac{1}{2}p\frac{y^2}{\mu^2}L^2_z T}e^{-\frac{y^2}{2}}dy \\
       &= \dbE\left[\Ind_{\{|\xi|\geq N\}}\right]^{\frac{1}{q}}\frac{1}{\sqrt{1-p\frac{L_z^2}{\mu^2}T}}, 
   \end{align*}
   which converges to $0$ as $N\to\infty$, with some $p>1$ close enough to 1 so that $\mu>pL_z\sqrt{T}$. 
 The assertion follows from the fact that the both terms on the right-hand side are independent of $\dbQ^\lambda$.
\end{remark}

We would also like to point out that Assumption \ref{assum:BSDExi} is closed to the condition (A4) in \cite{Bahlali2020} and the one in \cite[Proposition 6.1]{EH2011} which are used to establish the existence result of solutions to quadratic BSDEs.  

\begin{theorem} \label{thm:ExistenceUniquenessBSDE}
 Let Assumptions \ref{assum:BSDE.Lipschitz} and \ref{assum:BSDExi} hold true. 
 Then, the BSDE \eqref{eq:BSDE} has a unique solution $(Y,Z)\in \big(\cS^\beta(\PP)\cap \cD(\PP)\big) \times \cH^\beta(\PP)$ for all $\beta\in(0,1)$, with
 \begin{align} 
   \|Y\|_{\cD(\PP)}
      &\leq \cE^{\PP}\left[|\xi|+\int_0^T\big|f^0_s\big|ds\right], \label{BSDEEstimationEq1}\\
   \|Y\|_{\cS^\b(\PP)} + \|Z\|_{\cH^\b(\PP)}  
      &\leq C_{\beta,L,T}\left(\cE^{\PP}\big[|\xi|\big]^\beta+\cE^{\PP}\left[\int_0^T\big|f^0_s\big|ds \right]^{\beta}\right). \label{BSDEEstimationEq2}
  \end{align}
 for some constant $C_{\beta,L,T}$.
\end{theorem}

We also have the following comparison and stability results, which are direct consequences of the forthcoming Theorem \ref{thm:RBSDE.stability.comparison} and the estimates \eqref{BSDEEstimationEq1}-\eqref{BSDEEstimationEq2} of Theorem \ref{thm:ExistenceUniquenessBSDE}. 

\begin{theorem} \label{thm:BSDE.Stab.Comp}
 Let $(f,\xi)$ and $(f',\xi')$ satisfy the assumptions of Theorem \ref{thm:ExistenceUniquenessBSDE}, and $(Y,Z)$ and $(Y',Z')$ be the corresponding solutions.  
 \begin{enumerate}[$(i)$]
  \item \textnormal{Stability:} Denoting $\delta Y:=Y'-Y$, $\delta Y:=Z'-Z$ and $\delta\xi:=\xi'-\xi$, $\delta f:=f'-f$, we have for all $\beta\in(0,1)$, and some constant $C_{\beta,L,T}$:
  \begin{align*}
    \|\d Y\|_{\cD(\dbP)} 
       &\leq
       \cE^{\PP}\left[|\delta \xi|+\int_0^T\big|\delta f_s(Y_s,Z_s)\big|ds\right], \\
    \|\d Y\|_{\cS^\b(\PP)} + \| \d Z\|_{\cH^\b(\PP)} 
       &\leq
        C_{\beta,L,T}\left(\cE^{\PP}\big[|\delta\xi|\big]^\beta
                                     +\cE^{\PP}\left[\int_0^T|\delta f_s(Y_s,Z_s)|ds \right]^{\beta}
                     \right).       
  \end{align*}
  \item \textnormal{Comparison:} Suppose that $\xi\leq \xi'$, $\PP$-a.s., and $f(y,z)\leq f'(y,z)$, $dt\otimes\PP$-a.e., for all $(y,z)\in\RR\times\RR^d$.
                                 Then, $Y_\tau\leq Y'_\tau$, $\PP$-a.s., for all $\tau\in\cT_{0,T}$. 
 \end{enumerate}
\end{theorem}

\subsection{$\LL^1$-solution of reflected backward SDE}

Consider the following reflected backward stochastic differential equation (RBSDE) 
  \begin{eqnarray} \label{eq:RBSDE}
    Y_{t} =\xi + \int_{t}^T f_s(Y_s,Z_s)ds - \int_{t}^T Z_s\cdot dB_s + \int_{t}^T dK_s, \quad t\in[0,T], \quad \PP-a.s.  
  \end{eqnarray}
 where $Z$ is a progressively measurable $\RR^d$-valued process, $K$ is a nondecreasing process starting from $K_0=0$, and $Y$ is a scalar c\`adl\`ag adapted process satisfying the following Skorokhod condition with c\`adl\`ag obstacle $(S_t)_{0\leq t\leq T}$:
      \beaa
      Y_t ~\ge~ S_t,
      &t\in[0,T],~~\mbox{and}&
      \int_0^T (Y_{t-}-S_{t-})dK_t = 0, 
      ~~\PP\mbox{-a.s.}
      \eeaa 

Our second wellposedness result is the following.

\begin{theorem}  \label{RBSDEmainthm}
 Let Assumptions \ref{assum:BSDE.Lipschitz} and \ref{assum:BSDExi} hold true. Assume that $S^+\in\cD(\PP)$. Then, the RBSDE \eqref{eq:RBSDE} has a unique solution $(Y,Z,K)\in \big(\cS^\beta(\PP)\cap \cD(\PP)\big) \times \cH^\beta(\PP)\times \cI^\beta(\PP)$ for all $\beta\in(0,1)$. 
\end{theorem}

\begin{remark}
 Unlike BSDEs in Subsection \ref{subsection:BSDE}, we are only able to provide a partial a priori estimate, without an estimate for $Y$, due to the presence of the term $K$ and the exponent $\beta<1$, see Proposition \ref{EstimationZNK}. 
 We also note that a complete a priori estimate with $Y$ is possible if the generator $f$ does not depend on $z$ and the barrier is continuous, see e.g.~\cite[Proposition 5.1]{RS12}. 
\end{remark}

We also have the following stability and comparison results. 

\begin{theorem} \label{thm:RBSDE.stability.comparison}
  Let $(f,\xi,S)$ and $(f',\xi', S')$ satisfy the assumptions of Theorem \ref{RBSDEmainthm} with corresponding solutions $(Y,Z,K)$ and $(Y',Z',K')$.
  \begin{enumerate}[$(i)$]
   \item \textnormal{Stability:} Let $S=S'$. With $\delta Y:=Y'-Y$, $\delta Z:=Z'-Z$, $\delta K:=K'-K$, $\delta\xi:=\xi'-\xi$ and $\delta f:=f'-f$, we have
     \begin{align*}
       \|\d Y\|_{\cD(\PP)}
        &\leq
        \cE^{\PP}\left[|\delta\xi|+\int_0^T\big|\delta f_s(Y_s,Z_s)\big|ds\right],
     \end{align*}
     and for all $\beta\in(0,1)$, there exists a constant $C=C_{\beta,L,T}$ such that 
       \begin{align*}
        \|\d Y\|_{\cS^\b(\PP)} + \|\d Z\|_{\cH^\b(\PP)} + \|\delta K\|_{\cS^\beta(\PP)} 
        &\leq C\left(\cE^{\PP}\left[|\delta\xi|\right]^\beta + \cE^{\PP}\left[\int_0^T|\delta f_s(Y_s,Z_s)|ds\right]^\beta\right).
       \end{align*}
   \item \textnormal{Comparison:} Suppose that $\xi\leq\xi'$, $\PP$-a.s.; $f(y,z)\leq f'(y,z)$, $dt\otimes\PP$-a.e., for all $y,z\in\RR\times\RR^d$; 
                                    and $S\leq S'$, $dt\otimes\PP$-a.e.
                                  Then, $Y_\tau\leq Y'_\tau$, for all $\tau\in\cT_{0,T}$.  
  \end{enumerate}
\end{theorem}

\subsection{$\LL^1$-solution of second-order backward SDE}

Following Soner, Touzi \& Zhang \cite{STZ12}, we introduce second-order backward SDE as a family of backward SDEs defined on the supports of a convenient family of singular probability measures. 
We are given a fixed family $\{\cP(t,\omega)\}_{(t,\omega)\in[0,T]\times\Omega}$ of sets of probability measures on $(\Omega,\cF_T)$, where $\cP(t,\omega)\subseteq\overline{\cP}_S$. 
In order to apply the classic measurable selection results, we need the following properties of the probability families $\{\cP(t,\omega)\}_{(t,\omega)\in[0,T]\times\Omega}$. 

\begin{assumption}  \label{assumption:ppt-prob}
The family of probability measures $\cP=\{\cP(t,\omega)\}_{(t,\omega)\in[0,T]\times\Omega}$ verifies the following properties:
\begin{enumerate}[$(i)$]
 \item For every $(t,\omega)\in[0,T]\times\Omega$, one has $\cP(t,\omega)=\cP(t,\omega_{\cdot\wedge t})$ and $\PP(\Omega^\omega_t)=1$ whenever $\PP\in\cP(t,\omega)$.
       The graph $\llbracket\cP\rrbracket$ of $\cP$, defined by $\llbracket \cP\rrbracket:=\{(t,\omega,\PP)\,|\,\PP\in\cP(t,\omega)\}$, is analytic in $[0,T]\times\Omega\times\cM_1$.
 \item $\cP$ is stable under conditioning: i.e., for every $(t,\omega)\in[0,T]\times\Omega$ and every $\PP\in\cP(t,\omega)$ together with an $\FF$-stopping time $\tau$ taking values in $[t,T]$ and $\theta:=\tau^{t,\omega}-t$, 
       one has for $\PP$-a.e.~$\omega'\in\Omega$, $\PP^{\theta,\omega'}\in\cP\big(\tau(\omega\otimes_t\omega'),\omega\otimes_t\omega'\big)$.
 \item $\cP$ is stable under concatenation, i.e., let $\mathbb{P}\in \mathcal{P}(t, \omega)$, $\sigma$ and $\theta$ be defined as above,
       and let $\nu:\Omega\longrightarrow \mathcal{M}_1$ be an $\mathcal{F}_\theta$-measurable kernel with $\nu({\omega'})\in \mathcal{P}(\sigma, \omega\otimes_t\omega')$ for $\mathbb{P}$-a.e.~${\omega'}\in \Omega$, 
       then the measure $\overline{\PP}:=\mathbb{P}\otimes_\theta \nu$ defined by 
       \begin{equation}\label{concprob}
         \overline{\mathbb{P}}(A)=\int\int ({\bf 1}_A)^{\theta,\omega'}(\omega'')\nu(d\omega''; {\omega'})\mathbb{P}(d\omega'), \quad A\in \mathcal{F},
       \end{equation}
       is an element of $\mathcal{P}(t, \omega)$.
\end{enumerate}
\end{assumption}

Clearly, $\cP(0,\omega)$ is independent of $\omega$, and therefore $\cP_0:=\cP(0,{\bf 0})=\cP(0,\omega)$. 

\begin{remark} \label{rem:assummeasselection}
It is proved in \cite[Theorem 2.4, Corollary 2.5]{NN13}, that Assumption \ref{assumption:ppt-prob} holds in the context of the two following important situations (see also \cite[Remark 1.3]{PTZ18}):
\begin{itemize}
 \item the family $\{\overline{\cP}_S(t,\omega)\}$ is defined as a natural dynamic version of $\overline{\cP}_S$,
 \item the family $\{\cP_U(t,\omega)\}$ defined for some closed convex subset $U$ of $\dbS_d^+(\dbR)$ by $\cP_U(t,\omega):=\big\{\dbP\in\overline{\cP}_S(t,\omega):~\widehat{\sigma}\in U\big\}$.
\end{itemize}
\end{remark}

Our general 2BSDE takes the following form:
\begin{equation} \label{eq:2BSDE}
  Y_t = \xi + \int_t^Tf_s(Y_s,Z_s)ds - Z_s\cdot dB_s + dK_s,\quad \cP_0\mbox{-q.s.}
\end{equation}
 for some nondecreasing process $K$ satisfying $K_0=0$ and together with an appropriate minimality condition. 
We recall that a property is said to hold $\cP_0$-quasi surely, abbreviated as $\cP_0$-q.s., if it holds $\PP$-a.s.~for all $\PP\in\cP_0$. 

\begin{definition}\label{def0}
For $\beta\in(0,1)$, the process $(Y,Z)\in\cD\big(\cP_0,\FF^{+,\cP_0}\big)\times\cH^\beta\big(\cP_0,\FF^{\cP_0}\big)$ is a supersolution of the 2BSDE \eqref{eq:2BSDE}, if for all $\PP\in\cP_0$, the process
   $$
   K^\PP_t
   :=
   Y_0 - Y_{t} - \int_0^t F_s\big(Y_s, Z_s,\widehat\sigma_s\big)ds + Z_s \cdot dB_s,
   ~~t\in[0,T],
   ~~\PP\mbox{-a.s.}
   $$
is a predictable nondecreasing process in $\cI^\beta(\PP)$ with $K_0=0$.
\end{definition}

The dependence of $K^\PP$ on $\PP$ is inherited from the dependence of the stochastic integral $\int_0^. Z_s\cdot dB_s$ on the underlying local martingale measure $\PP$.\footnote{By Theorem 2.2 in Nutz \cite{Nut12}, the family $\{\int_0^. Z_s\cdot dB_s\}_{\dbP\in \cP_0}$ can be aggregated as a medial limit $\int_0^. Z_s\cdot dB_s$ under the acceptance of Zermelo-Fraenkel set theory with axiom of choice together with the continuum hypothesis into our framework. In this case, $\int_0^. Z_s\cdot dB_s$ can be chosen as an $\mathbb{F}^{+, \cP_0}$-adapted process, and the family $\{K^\dbP\}_{\dbP\in\cP_0}$ can be aggregated into the resulting medial limit $K$, i.e., $K=K^\dbP$, $\dbP$-a.s.~for all $\dbP\in\cP_0$. } Because of this the 2BSDE representation \eqref{eq:2BSDE} should be rather written under each $\PP\in\cP_0$ as:
	\begin{equation} \label{2bsdel}
           Y_{t} = \xi + \int_t^T F_s\big(Y_s, Z_s,\widehat\sigma_s\big)ds
                                           -Z_s \cdot dB_s
                                           + dK^\PP_s, 
            \quad\PP\mbox{-a.s.}
        \end{equation}

\begin{assumption} \label{AssumptionOmegaBSDE}
The terminal condition $\xi$ and the generator $F$ satisfy the integrability:
   \begin{align*}
     \lim_{n\to\infty}\left\| \big|\xi^{t,\omega}\big|\Ind_{\{|\xi^{t,\omega}|\geq n\}} 
                                            + \int_0^{T-t}\left|f^{0,t,\omega}_s\right|\Ind_{\{|f^{0,t,\omega}_s|\geq n\}}ds
                                   \right\|_{\cL^1(\cP(t,\omega))} = 0
                              \quad \mbox{for all} ~~
                              (t,\omega)\in[0,T]\times\Omega,
   \end{align*}
   where
    $$ f^{0,t,\om}_s(\om'):=F_{t+s}\big(\om\otimes_t\om',0,0,\widehat\sigma_s(\om')\big). $$
\end{assumption}

For all $\PP\in\cP_0$, we denote by $\big(\cY^{\PP},\cZ^{\PP}\big)$ the unique solution of the backward SDE \eqref{eq:BSDE}. By (\textbf{H1}), there exist two random fields $a^{\PP}(y,z)$ and $b^{\PP}(y,z)$ bounded by $L$  such that 
 $$ 
 f_s(y,z) - f_s\big(\cY^{\PP}_s,\cZ^{\PP}_s\big) 
 = a_s^{\PP}\big(y- \cY_s^{\PP} \big) + b_s^{\PP}\cdot\widehat\sigma_s\big(z- \cZ_s^{\PP} \big). 
 $$
We also define for all stopping times $\tau\in\cT_{0,T}$: 
 $$ \cP(\tau,\PP):= \big\{\PP'\in\cP_0:~\PP'=\PP~\mbox{on}~\cF_{\tau}\big\} \quad \mbox{ and } \quad \cP_+(\tau,\PP) := \bigcup_{h>0}\cP\big((\tau+h)\wedge T,\PP\big). $$
We now introduce our notion of second order backward SDE by means of a minimality condition involving the last function $b^\PP$.

\begin{definition}\label{def1}
For $\beta\in(0,1)$, the process $(Y,Z)\in\cD\big(\cP_0,\FF^{+,\cP_0}\big)\times\cH^\beta\big(\cP_0,\FF^{\cP_0}\big)$ is a solution to 2BSDE \eqref{eq:2BSDE} if it is a supersolution in the sense of Definition \ref{def0}, and it satisfies the minimality condition:
          \begin{equation}\label{mincond}
             K^{\mathbb{P}}_{\tau} = \essinf^\mathbb{P}_{\mathbb{P}'\in \mathcal{P}_+(\tau, \mathbb{P})}\mathbb{E}^{\mathbb{Q}_\tau^{\PP'}}\left[K^{\mathbb{P}'}_{T}\Big|\mathcal{F}^{+, \mathbb{P}'}_{\tau}\right], \quad \mathbb{P}\mbox{-a.s.}~~\mbox{for all}~~ \PP\in \cP_0,~~\tau\in\cT_{0,T},
          \end{equation}
where $\QQ_\tau^{\PP'}\in\cQ_{L_z}(\PP')$ is defined by the density $\frac{d\QQ_\tau^{\PP'}}{d\dbP'}:=\frac{G^{b^{\dbP'}\!\!(Y,Z)}_T}{G^{b^{\dbP'}\!\!(Y,Z)}_\tau}$.
\end{definition} 

Note that $\QQ^{\PP'}_\tau\big|_{\cF_\tau^+}=\PP'\big|_{\cF_\tau^+}=\PP\big|_{\cF_\tau^+}$ and the process $W_t-\int_\tau^tb_s^{\PP'}ds$ is a Brownian motion starting from $W_\tau$. 

Our next main result is the following wellposedness of second order backward SDEs, we emphasize again that Assumption \ref{assum:BSDE.Lipschitz} holds true in typical situations, see Remark \ref{rem:assummeasselection}.

\begin{theorem}
Under Assumptions \ref{assum:BSDE.Lipschitz}, \ref{assumption:ppt-prob} and \ref{AssumptionOmegaBSDE}, the 2BSDE \eqref{eq:2BSDE} has a unique solution $(Y,Z)\in\cD\big(\cP_0,\FF^{+,\cP_0}\big)\times\cH^\beta\big(\cP_0,\FF^{\cP_0}\big)$, for all $\beta\in(0,1)$. 

\end{theorem}

Similar to Soner, Touzi \& Zhang \cite{STZ12}, the following comparison result for second order backward SDEs is a by-product of our construction; the proof is provided in Theorem \ref{thm:representation}.

\begin{proposition}\label{prop:2BSDEcomp}
Let $(Y,Z)$ and $(Y',Z')$ be solutions of 2BSDEs with parameters $(F,\xi)$ and $(F',\xi')$, respectively, which satisfy Assumptions \ref{assum:BSDE.Lipschitz} and \ref{AssumptionOmegaBSDE}. 
Suppose further that $\xi\leq\xi'$ and $F_t\big(y,z,\widehat\sigma_t\big)\le F'_t\big(y,z,\widehat\sigma_t\big)$ for all $(y,z)\in\dbR\times\dbR^d$, $dt\otimes\cP_0$-q.s. 
Then, we have $Y\leq Y'$, $dt\otimes\cP_0$-q.s.
\end{proposition}

\section{Wellposedness of reflected BSDEs}  \label{sect:RBSDE}

Throughout this section, we fix a probability measure $\PP\in\overline{\cP}_S$, and we omit the dependence on $\PP$ in all of our notations (e.g.~$\cD(\PP)$ denoted as $\cD$). 
It is clear from Girsanov's Theorem that under a measure $\QQ^\lambda\in\cQ_{L_z}$, the process $W^{\lambda}:=W-\int_0^{\cdot}\lambda_sds$ is a Brownian motion under $\QQ^\lambda$. 

\begin{remark}
 We note that under a measure $\QQ^\lambda\in\cQ_{L_z}$ defined as above, the RBSDE satisfies 
  \begin{equation*}
    dY_t = -\big(f_t(Y_t,Z_t)-\widehat\sigma^\intercal_tZ_t\cdot\lambda_t\big)dt + Z_t\cdot dB^\lambda_t -dK_t, 
  \end{equation*}
  where the process 
   $ B_t^\lambda:= B_t-\int_0^t\widehat\sigma_s\lambda_sds$
  is a local martingale under $\QQ^\lambda$, and the generator $f_t(y,z)- \widehat\sigma_t^\intercal z\cdot\lambda_t$ satisfies the Assumption \ref{assum:BSDE.Lipschitz} with Lipschitz coefficients $L_y$ and $2L_z$. 
\end{remark}

\subsection{Some useful inequalities}

First of all, we provide an estimation of a running supremum process. 

\begin{lemma} \label{ELbetaleqsupEL}
 Let $X$ be a nonnegative c\`adl\`ag process, and $X^*_t:=\max_{s\le t} X_s$. Then, 
  \begin{align*}
    \EE\big[(X_T^{*})^\beta\big]
      \leq
    \frac{1}{1-\beta}
    \sup_{\tau\in\cT_{0,T}}\EE[X_\tau]^\beta, 
     \quad \mbox{for all}~~
    \beta\in(0,1).
  \end{align*}
\end{lemma}

\begin{proof}
 For $x>0$, define $\tau_x:=\inf\{t>0: X_t\geq x\}$, with the usual convention that $\inf\emptyset = \infty$.  
 We have that $X^{*}_{\tau_x}=X_{\tau_x}$, and 
    $$ 
    \PP[\tau_x\leq T] 
    \;=\; 
    \PP[X_{\tau_x}\geq x] 
    \;\le\; 
    \frac{\EE[X_{\tau_x}]}{x}
    \;\le\; 
    \frac{c}{x}, 
    $$ 
   with $c:=\sup_{\tau\in\cT_{0,T}}\EE[X_\tau]$,    
   which implies that $ \PP[\tau_x\leq T] \leq \frac{c}{x}\wedge 1$.
 Then,  for $\beta\in(0,1)$
 \begin{equation*}
  \begin{aligned}
    \EE\left[(X^{*}_T)^\beta\right] &= \EE\left[\int_0^\infty\Ind_{\big\{X^{*}_T\geq x\big\}}\beta x^{\beta-1}dx\right] = \int_0^\infty\PP\left[X^{*}_T\geq x\right]\beta x^{\beta-1}dx \\
      & 
      = \int_0^\infty\PP\left[\tau_x\leq T\right]\beta x^{\beta-1}dx 
      \;\le\; \int_0^\infty \left(\frac{c}{x}\wedge 1\right)\beta x^{\beta-1}dx 
      \;=\; \frac{c^\beta}{1-\beta}.
  \end{aligned}
 \end{equation*} 
\end{proof}

\begin{lemma}  \label{techLemma1}
  Let $\zeta$ be a nonnegative $\cF_T$-measurable ramdon variable and $Y$ a nonnegative process such that
   \begin{eqnarray} \label{YleqEX}
     \sup_{\tau\in\cT_{0,T}}\left\{Y_\tau - \EE^{\widehat\QQ}\big[\zeta\big|\cF_\tau^{+,\PP}\big]\right\}\leq 0,
     &\mbox{for some}& 
     \widehat\QQ\in\cQ_{L_z}.
   \end{eqnarray}
  Then, $ \sup_{\tau\in\cT_{0,T}}\cE^{\PP}[ Y_{\tau} ] \leq \cE^{\PP}[\zeta]$.
\end{lemma}

\begin{proof}
  Fix $\QQ\in\cQ_{L_z}$ and $\tau\in\cT_{0,T}$.
  Notice that $\QQ\otimes_\tau\widehat{\QQ}\in\cQ_{L_z}$ for all $\QQ\in\cQ_{L_z}$. 
  Then, it follows from \eqref{YleqEX} that
   $$ \EE^{\QQ}[Y_\tau]\leq \EE^{\QQ}\left[\EE^{\widehat\QQ}\big[\zeta\big|\cF^{+,\dbP}_\tau\big]\right] = \EE^{\QQ\otimes_\tau\widehat{\QQ}}[\zeta] \leq \cE^{\PP}[\zeta]. $$
  The required inequality follows by taking supremum over all stopping times and $\QQ\in\cQ_{L_z}$.
\end{proof}

Now, we show a Doob-type inequality under the nonlinear expectation $\cE^{\PP}$, which turns out to be crucial for our analysis.

\begin{lemma}  \label{supDoob}
 Let $(M_t)_{0\leq t\leq T}$ be a nonnegative submartingale under some $\widehat{\QQ}\in\cQ_{L_z}$. 
 Then, 
  \beaa
    \|M\|_{\cS^\b}  
    \;\le\; 
    \frac{1}{1-\beta}\;\cE^{\PP}[M_T]^{\beta}
    &\mbox{for all}&
    0<\beta<1.
  \eeaa
\end{lemma}
  
 \begin{proof}
   Let $x>0$ and $\QQ\in\cQ_{L_z}$ be arbitrary. Define 
     $\tau_x:=\inf\left\{t\geq 0\,\big|\, M_t\geq x\right\}$,
    with the usual convention that $\inf\emptyset =\infty$. 
   From the optional sampling theorem and the definition of concatenation, we obtain that
    \begin{equation*}
      \begin{aligned}
        \EE^\QQ\left[M_{\tau_x\wedge T}\right]  
                 \leq \EE^{\QQ}\left[\EE^{\widehat\QQ}\big[M_T\big|\cF_{\tau_x\wedge T}^{+,\PP}\big]\right] 
                 = \EE^{\QQ\otimes_{\tau_x\wedge T}\widehat\QQ}\left[M_T\right],
      \end{aligned}
    \end{equation*}
    for each $\QQ\in\cQ_{L_z}$. 
   As $\QQ\otimes_{\tau_x\wedge T}\widehat\QQ\in\cQ_{L_z}$, this provides that $ \EE^\QQ\left[M_{\tau_x\wedge T}\right] \leq \cE^\PP\left[M_{T}\right]=:c$.
   
   Let us denote $M_*:=\sup_{0\leq t\leq T}M_t.$
   It follows that
    $$ x\QQ\left[M_*\geq x\right] = x\QQ[\tau_x\leq T] \leq \EE^\QQ\big[M_{\tau_x}\Ind_{\{\tau_x\leq T\}}\big] \leq \EE^\QQ\big[M_{\tau_x\wedge T}\big]\leq c. $$
   Then, 
     \begin{equation*}
      \begin{aligned}
       \EE^\QQ\big[M_*^\beta\big] 
        = \int_0^\infty\QQ\left[M_*\geq x\right]\beta x^{\beta-1}dx  
        \leq \int_0^\infty\left(1\wedge\frac{c}{x}\right)\beta x^{\beta-1}dx = \frac{c^\beta}{1-\beta}.       
      \end{aligned}
     \end{equation*}
   As $\QQ\in\cQ_{L_z}$ is arbitrary, the assertion follows. 
 \end{proof}
 
\subsection{A priori estimates for reflected backward SDEs}

We will construct a solution of the RBSDE \eqref{eq:RBSDE}, using a sequence of $\LL^2$-solutions to the related RBSDEs. 
The following a priori estimation is crucial for the existence result. 

\begin{proposition} \label{EstimationZNK}
 Let $(Y,Z,K)$ be a solution of RBSDE \eqref{eq:RBSDE}. 
 For all $\beta\in(0,1)$, there exists a constant $C_{\beta,L,T}>0$ such that
  \begin{align*}
     \|Z\|_{\cH^\b} + \|K\|_{\cS^\b}
      &\leq C_{\beta,L,T}\bigg(\|Y\|_{\cS^\b} + \|Y\|_{\cD}^\beta + \cE^{\PP}\bigg[\int_0^T\big|f^0_s\big|ds \bigg]^{\beta}\bigg). 
  \end{align*}
\end{proposition}

\begin{proof}
 \noindent \textbf{Step 1.} We first derive the following estimate of $K$: 
  \begin{align} \label{estimationKeq0}
    \|K\|_{\cS^\b} 
      \leq 
    C^K_{\beta,L,T}\bigg(\|Y\|_{\cD}^\beta + \cE^{\PP}\bigg[\int_0^T\big|f^0_s\big|ds\bigg]^\beta\bigg),
  \end{align}
  where $C^K_{\beta,L,T}$ is a positive constant depending on $\beta, L_y, L_z$ and $T$. 
  
  Indeed, it follows from \eqref{eq:RBSDE} and Assumption \ref{assum:BSDE.Lipschitz} that 
    \begin{align*}
       K_t &\leq |Y_0| + |Y_t| + \int_0^t\big|f^0_s\big|ds + L_y\int_0^t|Y_s|ds + L_z\int_0^t\big|\widehat\sigma_s^\intercal Z_s\big|ds + \int_0^t\widehat\sigma_s^{\intercal}Z_s\cdot dW_s.
    \end{align*}
  Define
   \begin{align} \label{speciallambda}
     \lambda_s:=L_z\frac{\widehat\sigma_s^{\intercal}Z_s}{|\widehat\sigma_s^{\intercal}Z_s|}\Ind_{\{|\widehat\sigma_s^{\intercal}Z_s|\neq 0\}}
          \q\mbox{and}\q
      \frac{d\QQ^\lambda}{d\PP}:= \exp\left\{-\int_0^T\lambda_s\cdot dW_s - \frac{1}{2}\int_0^T|\lambda_s|^2ds\right\}. 
   \end{align}
  Let $\{\tau_k\}_{k\in\NN}$ be a localizing sequence of stopping times, such that the stochastic integral is a martingale. 
  We have 
   \begin{align*}
    K_{\tau_k} &\leq |Y_0| + |Y_{\tau_k}| + \int_0^T|f^0_s|ds + L_y\int_0^T|Y_s|ds + \int_0^{\tau_k}\widehat\sigma_s^{\intercal}Z_s\cdot \left(dW_s + \lambda_sds\right).
   \end{align*}
 Taking expectation with respect to $\QQ^{\lambda}$ we obtain 
  \begin{align*}
   \EE^{\QQ^{\lambda}}[K_{\tau_k}] \leq (2+L_yT)\sup_{\tau\in\cT_{0,T}}\EE^{\QQ^{\lambda}}[|Y_\tau|] + \EE^{\QQ^{\lambda}}\left[\int_0^T\big|f^0_s\big|ds\right].
  \end{align*}
 It follows by monotone convergence theorem that 
  \begin{align}  \label{controlKeq1}
    \EE^{\QQ^{\lambda}}[K_T] \leq (2+L_yT)\sup_{\tau\in\cT_{0,T}}\cE^{\PP}[|Y_\tau|] + \cE^{\PP}\left[\int_0^T\big|f^0_s\big|ds\right]<\infty. 
  \end{align}
 Now take any $\|\lambda'\|_\infty\le L_z$. By the Girsanov transformation and the H\"older inequality with $\frac{1}{q}=1-\beta$, we have
  \begin{align*}
      & \EE^{\QQ^{\lambda'}} \left[K_T^\beta\right] = \EE^{\QQ^{\lambda}}\left[\frac{d\QQ^{\lambda'}}{d\PP}\left(\frac{d\QQ^{\lambda}}{d\PP}\right)^{-1}K_T^\beta\right] 
         \leq \EE^{\QQ^{\lambda}}\left[\left(\frac{d\QQ^{\lambda'}}{d\PP}\left(\frac{d\QQ^{\lambda}}{d\PP}\right)^{-1}\right)^q\right]^{\frac{1}{q}}\EE^{\QQ^{\lambda}}[K_T]^\beta \\
      & = \EE^{\QQ^{\lambda}}\left[\exp\left(\int_0^Tq(\lambda_s-\lambda'_s)\cdot dW_s - \frac{q}{2}\int_0^T\big(|\lambda_s|^2-|\lambda'_s|^2\big)ds\right)\right]^{\frac{1}{q}}\EE^{\QQ^{\lambda}}[K_T]^\beta  \\
      & = \EE^{\PP}\left[\exp\left(\int_0^T\big((q-1)\lambda_s-q\lambda'_s\big)\cdot dW_s - \frac{1}{2}\int_0^T\big((q-1)|\lambda_s|^2-q|\lambda'_s|^2\big)ds\right)\right]^{\frac{1}{q}}\EE^{\QQ^{\lambda}}[K_T]^\beta  \\
      & = \EE^{\PP}\Big[\exp\Big(\int_0^T\big((q-1)\lambda_s-q\lambda'_s\big)\cdot dW_s -\frac{1}{2}\int_0^T\big|(q-1)\lambda_s-q\lambda'_s\big|^2ds\Big) \\
      & \hspace{12mm} \cdot \exp\Big(\frac{1}{2}\int_0^T\underbrace{\big|(q-1)\lambda_s-q\lambda'_s\big|^2}_{\leq 4q^2L_z^2}ds 
                              - \frac{1}{2}\int_0^T\underbrace{\big((q-1)|\lambda_s|^2-q|\lambda'_s|^2\big)}_{\leq 2q L_z^2\leq 2q^2L_z^2}ds \Big)\Big]^{\frac{1}{q}}\EE^{\QQ^{\lambda}}[K_T]^\beta  \\
      & \leq \exp\big(3qL_z^2T\big)\EE^{\QQ^{\lambda}}[K_T]^\beta.
  \end{align*}
 As $\lambda'$ is arbitrary, together with \eqref{controlKeq1} we obtain \eqref{estimationKeq0}.

\vspace{3mm}  
 
\noindent \textbf{Step 2.} We next estimate the stochastic integral $\int_0^TY_{s-}\big(\widehat\sigma_s^{\intercal}Z_s\cdot dW_s^\lambda-dK_s\big)$.
 Using Burkholder-Davis-Gundy inequality and Young's inequality, we obtain
 \allowdisplaybreaks
 \begin{align}  \label{int_0^TY_sdM^lambda_s}
    \EE^{\QQ^\lambda} & \left[\left|\int_0^T Y_{s-}\big(\widehat\sigma_s^{\intercal}Z_s\cdot dW_s^\lambda-dK_s\big)\right|^{\frac{\beta}{2}}\right]\nonumber \\
      & \leq \EE^{\QQ^\lambda}\left[\sup_{0\leq u\leq T}\left|\int_0^u Y_{s-}\widehat\sigma_s^{\intercal}Z_s\cdot dW_s^\lambda\right|^{\frac{\beta}{2}} + \left|\int_0^T Y_{s-}dK_s\right|^{\frac{\beta}{2}}\right] \nonumber \\
   & \leq C_\beta \EE^{\QQ^\lambda}\left[\left(\int_0^TY_{s-}^2\big|\widehat\sigma_s^{\intercal} Z_s\big|^2ds\right)^{\frac{\beta}{4}}\right] 
          + \EE^{\QQ^\lambda}\left[\sup_{0\leq s\leq T}|Y_{s}|^{\frac{\beta}{2}}K_T^{\frac{\beta}{2}}\right] \nonumber \\
   & \leq C_\beta \EE^{\QQ^\lambda}\left[\sup_{0\leq s\leq T} |Y_s|^{\frac{\beta}{2}}\left(\int_0^T\big|\widehat\sigma_s^{\intercal}Z_s\big|^2ds\right)^{\frac{\beta}{4}}\right] 
          + \EE^{\QQ^\lambda}\left[\sup_{0\leq s\leq T}|Y_{s}|^{\frac{\beta}{2}}K_T^{\frac{\beta}{2}}\right] \nonumber \\ 
   & \leq \frac{C_\beta+1}{2\varepsilon}\|Y\|_{\SSS^{\beta}(\dbQ^\lambda)} + \frac{C_\beta\varepsilon}{2}\|Z\|_{\HH^\beta(\dbQ^\lambda)} + \frac{\varepsilon}{2}\|K\|_{\II^\beta(\dbQ^\lambda)}.
 \end{align}
 
\vspace{3mm}

\noindent \textbf{Step 3.}  
 We now show that
 \begin{align}  \label{eq:EstimateZ}
   \|Z\|_{\cH^\b} \leq C^Z_{L,T,\beta}\left(\|Y\|_{\cS^\b}+ \|Y\|_{\cD}^\beta + \cE^{\PP}\left[\int_0^T\big|f^0_s\big|ds\right]^\beta\right).
 \end{align}

 Under a measure $\dbQ^\lambda\in \cQ_{L_z}$, the RBSDE \eqref{eq:RBSDE} satisfies
  $$ dY_t = -\big(f_t(Y_t,Z_t)-\widehat\sigma^{\intercal}_tZ_t\cdot\lambda_t\big)dt + \widehat\sigma_t^{\intercal}Z_t\cdot dW^\lambda_t -dK_t,\quad \dbQ^\lambda\mbox{-a.s.} $$
 Applying It\^o's formula on $Y^2$, we obtain 
  \begin{align*}
    \int_0^T \big|\widehat\sigma_s^{\intercal}Z_s\big|^2ds + \sum_{0< s\leq T}(\Delta K_s)^2
      = \xi^2 -Y^2_0 &+ \int_0^T 2 Y_{s-}\big(f_s(Y_s,Z_s)-\widehat\sigma_s^{\intercal}Z_s\cdot\lambda_s\big)ds \\
     &- \int_0^T 2Y_{s-}\left(\widehat\sigma_s^{\intercal}Z_s\cdot dW^\lambda_s - dK_s\right).
  \end{align*}
 Therefore, by Assumption \ref{assum:BSDE.Lipschitz} and Young's inequality, we obtain
 \allowdisplaybreaks
 \begin{align*}
  \frac12 &\int_0^T\big|\widehat\sigma_s^{\intercal}Z_s\big|^2ds \\ 
    &\leq (3+2L_yT+8L_z^2T)\sup_{0\leq s\leq T}Y_s^2 + \left(\int_0^T\big|f^0_s\big|ds\right)^2 + 2\left|\int_0^T Y_{s-}\big(\widehat\sigma_s^{\intercal}Z_s\cdot dW_s^\lambda-dK_s\big)\right|.
 \end{align*}
 Together with \eqref{int_0^TY_sdM^lambda_s}, we have for $\beta\in(0,1)$
  \begin{align} 
    & \|Z\|_{\HH^\beta(\QQ^\lambda)} \nonumber \\
        &\leq C_{L,T,\beta}\Bigg(\|Y\|_{\SSS^\beta(\QQ^\lambda)} + \EE^{\QQ^{\lambda}}\left[\left(\int_0^T\big|f^0_s\big|ds\right)^{\beta}\right]
                    +\EE^{\QQ^{\lambda}}\Bigg[\left|\int_0^T Y_{s-}\left(\widehat\sigma_s^{\intercal}Z_s\cdot dW^\lambda_s - dK_s\right)\right|^{\frac{\beta}{2}}\Bigg]\Bigg)  \nonumber \\
    & \leq C_{L,T,\beta}\left(\frac{C_\beta+1}{2\varepsilon}+1\right)\|Y\|_{\SSS^\beta(\QQ^\lambda)} + C_{L,T,\beta}\EE^{\QQ^{\lambda}}\left[\left(\int_0^T\big|f^0_s\big|ds\right)^{\beta}\right] \nonumber \\
        & \hspace{6mm} + \frac{C_{L,T,\beta}C_\beta\varepsilon}{2}\big\|Z\big\|_{\HH^\beta(\QQ^\lambda)} + \frac{C_{L,T,\beta}\varepsilon}{2}\|K\|_{\II^\beta(\QQ^\lambda)}.   \nonumber          
  \end{align}
  Choosing $\varepsilon=\frac{1}{C_{L,T,\beta}C_\beta}$, we get 
   \begin{align*}
     \|Z\|_{\HH^\beta(\QQ^\lambda)} \leq C'_{L,T,\beta}\left(\|Y\|_{\SSS^\beta(\QQ^\lambda)} + \EE^{\QQ^{\lambda}}\left[\left(\int_0^T\big|f^0_s\big|ds\right)^{\beta}\right] + \|K\|_{\II^\beta(\QQ^\lambda)}\right).
   \end{align*}
  Together with \eqref{estimationKeq0} we obtain \eqref{eq:EstimateZ}. 
\end{proof}

\subsection{Existence and Uniqueness}

We apply the existing well-posedness result of $\LL^2$ solutions of reflected BSDEs to establish the well-posedness result of our $\LL^1$ solutions of reflected BSDEs. 
In the first step, we show the well-posedness of RBSDEs with a $\LL^2$ barrier and $\LL^1$ generator $f$ and terminal data $\xi$ by an approximation with solutions of RBSDEs with $f^n$ and $\xi^n$ satisfying $\LL^2$ condition. 
In the second step, we prove our existence and uniqueness of $\LL^1$ solutions with $\LL^1$ barrier by an approximation with solutions of RBSDEs with $\LL^2$ barriers and $\LL^1$ generator and terminal data obtained from step 1. 
As a special case $S=-\infty$ we obtain the existence and uniqueness result of $\LL^1$-solutions of reflected BSDEs.  

\subsubsection{Square integrable obstacle}

\begin{theorem}  \label{ExistenceUniqueness112}
 Let Assumption \ref{assum:BSDE.Lipschitz} hold true. Assume that $S^+\in\cS^2$, 
   then Theorem \ref{RBSDEmainthm} holds true. 
\end{theorem}

\begin{proof}
\textbf{Existence:}  
For each $n\in\NN$, we denote $\xi^n:=q_n(\xi)$ and $f^n_t(y,z):=f_t(y,z)-f_t^0+q_n(f^0_t)$, where $q_n(x):=\frac{xn}{|x|\vee n}$. 
As $S^+\in\cS^2$,  
 by \cite[Theorem 3.1]{BPTZ16}, RBSDE$(f^n,\xi^n,S)$ has a unique solution $(Y^n,Z^n,K^n)\in\SSS^2\times\HH^2\times\II^2$, and $Y^n$ belongs to class $\DD(\QQ)$ for each $\QQ\in\cQ_{L_z}$.
 
\ms
 
\noindent \textbf{Step 1:} We are going to show that $\{Y^n\}_{n\in\NN}$ is a Cauchy sequence in $\cS^\beta$ and $\cD$.  
Let $m,n\in\NN$ and $n\geq m$. Set $ \delta Y:=Y^n-Y^m$, $\delta Z:=Z^n-Z^m$, $\delta K:=K^n-K^m$ and $\delta f:=f^n-f^m$. 
Clearly, the process $(\delta Y,\delta Z,\delta K)$ satisfies the following equation 
 \begin{equation}  \label{deltaRBSDE}
   \delta Y_t = \delta\xi + \int_t^Tg_s(\delta Y_s,\delta Z_s)ds - \int_t^T\delta Z_s\cdot dB_s + \int_t^Td\delta K_s,
 \end{equation}
 where 
 $$ g_s(\delta Y_s,\delta Z_s):= f^n_s(Y_s^m+\delta Y_s,Z_s^m+\delta Z_s) - f^m_s(Y^m_s,Z^m_s). $$
It follows by Proposition 4.2 in \cite{LRTY20} that 
 \begin{align*}
   |\delta Y_\tau| & \leq |\delta Y_T| - \int_\tau^T\sgn(\delta Y_{s-})d\delta Y_s  \nonumber \\
     & = |\delta \xi| 
           + \int_\tau^T\sgn(\delta Y_{s})\big\{ g_s(\delta Y_s,\delta Z_s)ds - \widehat\sigma_s^{\intercal}\delta Z_s\cdot dW_s\big\} + \int_\tau^T\sgn(\delta Y_{s-})d\delta K_s.
 \end{align*}
By Assumption \ref{assum:BSDE.Lipschitz} and Remark \ref{rem:monotonicity} we obtain 
 \begin{align*}
  \sgn(\delta Y_{s})g_s(\delta Y_s,\delta Z_s) & = \sgn(\delta Y_{s})\big(f^n_s(Y_s^m+\delta Y_s,Z_s^m+\delta Z_s) - f^n_s(Y^m_s,Z^m_s+\delta Z_s)\big)  \nonumber \\
    & \hspace{5mm} + \sgn(\delta Y_{s})\big(f^n_s(Y_s^m,Z_s^m+\delta Z_s)-f^n_s(Y_s^m,Z_s^m)\big) \nonumber \\
    & \hspace{5mm} + \sgn(\delta Y_{s})\big(f^n_s(Y_s^m,Z_s^m)-f^m_s(Y_s^m,Z_s^m)\big) \nonumber \\
    & \le\; L_z\big|\widehat\sigma^{\intercal}_s\delta Z_s\big| + \big|\delta f_s(Y^m_s,Z^m_s)\big|  
    \; \le\; 
    L_z\big|\widehat\sigma^{\intercal}_s\delta Z_s\big| + \big|f^0_s\big|\Ind_{\{|f^0_s|\geq n\}}. 
 \end{align*}
We note that by Skorokhod condition, 
  \begin{align} \label{deltaYddeltaKleq0}
    \delta Y_{s-} d\delta K_s 
      & =(Y^{n}_{s-}-S_{s-})dK^{n}_s -(Y^{n}_{s-}-S_{s-})dK^m_s - (Y^m_{s-}-S_{s-})dK^{n}_s  + (Y^m_{s-}-S_{s-})dK^m_s  \nonumber \\
      & =-(Y^{n}_{s-}-S_{s-})dK^m_s - (Y^m_{s-}-S_{s-})dK^{n}_s ~ \leq 0, 
  \end{align}
 and 
  \begin{align}
    \sgn(\delta Y_{s-})d\delta K_s = \frac{\sgn(\delta Y_{s-})}{\delta Y_{s-}}\delta Y_{s-}d\delta K_s \leq 0.
  \end{align}
Therefore, we obtain
  \allowdisplaybreaks
 \begin{align*}
  |\delta Y_\tau| 
  & \leq |\xi|\Ind_{\{|\xi|\geq n\}} + \int_\tau^T\big(L_z\big|\widehat\sigma^{\intercal}_s\delta Z_s\big| + \big|f^0_s\big|\Ind_{\{|f^0_s|\geq n\}}\big)ds - \int_\tau^T\sgn(\delta Y_{s})\widehat\sigma_s^{\intercal}\delta Z_s\cdot dW_s \nonumber \\
  & = |\xi|\Ind_{\{|\xi|\geq n\}} + \int_\tau^T \big|f^0_s\big|\Ind_{\{|f^0_s|\geq n\}}ds 
                                  - \int_\tau^T \sgn(\delta Y_{s})\widehat\sigma_s^{\intercal}\delta Z_s\cdot\big(dW_s-\widehat\lambda_sds\big),
 \end{align*}
where $\widehat\lambda:=L_z\sgn(\delta Y_s)\frac{\widehat\sigma_s^{\intercal}\delta Z_s}{|\widehat\sigma_s^{\intercal}\delta Z_s|}\Ind_{\{|\widehat\sigma_s^{\intercal}\delta Z_s|\neq 0\}}$. 
As $\delta Z\in\cH^2$ and $\delta N\in\cN^2$, we deduce from Burkholder-Davis-Gundy inequality that the last term is a uniformly integrable martingale under the measure $\widehat\QQ\in\cQ_{L_z}$ 
  such that $\frac{d\widehat{\QQ}}{d\PP}:=G^{\widehat\lambda}_T. $
Taking conditional expectation with respect to $\cF_\tau^{+,\PP}$ under the equivalent measure $\widehat\QQ$, we obtain that 
 $$ |\delta Y_\tau| \leq \EE^{\widehat\QQ}\left[|\xi|\Ind_{\{|\xi|\geq n\}} + \int_0^T\big|f^0_s\big|\Ind_{\{|f^0_s|\geq n\}}ds\bigg|\cF_\tau^{+,\PP}\right]. $$
We deduce immediately from Lemma \ref{techLemma1} that 
 \begin{equation}
   \|\delta Y\|_{\cD}
   \leq \cE^{\PP}\left[|\xi|\Ind_{\{|\xi|\geq n\}} + \int_0^T\big|f^0_s\big|\Ind_{\{|f^0_s|\geq n\}}ds\right], 
 \end{equation}
 and from Lemma \ref{supDoob} that for any $\beta\in(0,1)$, 
  \begin{equation}
   \|\delta Y\|_{\cS^\beta} 
    \leq \frac{1}{1-\beta}\cE^{\PP}\left[|\xi|\Ind_{\{|\xi|\geq n\}} + \int_0^T\big|f^0_s\big|\Ind_{\{|f^0_s|\geq n\}}ds\right]^\beta.
  \end{equation}
This shows that $\{Y^n\}_{n\in\NN}$ is a Cauchy sequence in $\cD$ and $\cS^\beta$. By completeness of $\cD$ and $\cS^\beta$, there exists a limit $Y\in\cD\cap\cS^\beta$.

\ms

\noindent \textbf{Step 2:} We prove that $\{Z^n\}_{n\in\NN}$ is a Cauchy sequence in $\cH^\beta$.
By It\^o's formula, we have 
 \begin{align}
  (\delta Y_T)^2 - (\delta Y_0)^2 
    & = -2\int_0^T\delta Y_{s-}g_s(\delta Y_s,\delta Z_s)ds + 2\int_0^T\delta Y_{s-}\widehat\sigma_s^{\intercal}\delta Z_s\cdot dW_s   \nonumber \\
    & \quad -2\int_0^T\delta Y_{s-}d\delta K_s + \int_0^T\big|\widehat\sigma_s^{\intercal}\delta Z_s\big|^2ds + [\delta K]_T,  \nonumber 
 \end{align}
 and therefore by Assumption \ref{assum:BSDE.Lipschitz}, Skorokhod condition \eqref{deltaYddeltaKleq0} and Young's inequality
  \allowdisplaybreaks
 \begin{align}
  \int_0^T \big|\widehat\sigma_s^{\intercal}\delta Z_s\big|^2ds 
     &\leq \sup_{0\leq s\leq T}(\delta Y_s)^2 + 2\int_0^T |\delta Y_s|\big|f^0_s\big|\Ind_{\{|f^0_s|\geq n\}}ds + 2L_z\int_0^T|\delta Y_s|\big|\widehat\sigma_s^{\intercal}\delta Z_s\big|ds \nonumber \\
       & \hspace{25mm} - 2\int_0^T\delta Y_{s} \widehat\sigma_s^{\intercal}\delta Z_s\cdot dW_s  \nonumber  \\
     &\leq \sup_{0\leq s\leq T}(\delta Y_s)^2 + 2\int_0^T|\delta Y_s|\big|f^0_s\big|\Ind_{\{|f^0_s|\geq n\}}ds + 4L_z\int_0^T|\delta Y_s|\big|\widehat\sigma_s^{\intercal}\delta Z_s\big|ds \nonumber \\
       & \hspace{25mm} + 2\sup_{0\leq u\leq T}\left|\int_0^u\delta Y_{s}\widehat\sigma_s^{\intercal}\delta Z_s\cdot dW^\lambda_s\right|   \nonumber  \\
     &\leq (2+8L_z^2T)\sup_{0\leq s\leq T}(\delta Y_s)^2 + \left(\int_0^T\big|f^0_s\big|\Ind_{\{|f^0_s|\geq n\}}ds\right)^2 + \frac{1}{2}\int_0^T\big|\widehat\sigma_s^{\intercal}\delta Z_s\big|^2ds  \nonumber \\
     & \hspace{25mm} + 2\sup_{0\leq u\leq T}\left|\int_0^u\delta Y_{s}\widehat\sigma_s^{\intercal}\delta Z_s\cdot dW^\lambda_s\right|,   \nonumber 
 \end{align}
 which implies that 
 \begin{align}\label{itosameS}
  \int_0^T & \big|\widehat\sigma_s^{\intercal}\delta Z_s\big|^2ds \leq 2(2+8L_z^2T)\sup_{0\leq s\leq T}(\delta Y_s)^2 + 2\left(\int_0^T\big|f^0_s\big|\Ind_{\{|f^0_s|\geq n\}}ds\right)^2 \nonumber \\
    & \hspace{25mm} + 4\sup_{0\leq u\leq T}\left|\int_0^u\delta Y_{s}\widehat\sigma_s^{\intercal}\delta Z_s\cdot dW^\lambda_s\right|.
 \end{align}
 Taking expectation we obtain 
  \begin{align*}
    \|\d Z\|_{\dbH^\b(\QQ^\lambda)}   &\leq C'_{\beta,L,T}\Bigg(\|\d Y\|_{\dbS^\b(\QQ^\lambda)} + \EE^{\QQ^\lambda}\left[\left(\int_0^T\big|f^0_s\big|\Ind_{\{|f^0_s|\geq n\}}ds\right)^\beta\right] \nonumber \\
      & \hspace{37mm} + \EE^{\QQ^\lambda}\left[\sup_{0\leq u\leq T}\left|\int_0^u\delta Y_{s}\widehat\sigma_s^{\intercal}\delta Z_s\cdot dW^\lambda_s\right|^{\frac{\beta}{2}}\right]\Bigg).
  \end{align*}
 By Burkholder-Davis-Gundy inequality, we have
  \begin{align*}
   \EE^{\QQ^\lambda} \left[\sup_{0\leq u\leq T}\left|\int_0^u\delta Y_{s}\widehat\sigma_s^{\intercal}\delta Z_s\cdot dW^\lambda_s\right|^{\frac{\beta}{2}}\right]
      & \leq C'_\beta \EE^{\QQ^\lambda}\left[\left(\int_0^T|\delta Y_s|^2\big|\widehat\sigma_s^{\intercal}\delta Z_s\big|^2ds\right)^{\frac{\beta}{4}}\right] \nonumber \\
      & \leq \frac{C'_\beta}{2\varepsilon} \|\d Y\|_{\dbS^\b(\QQ^\lambda)} + \frac{C'_\beta\varepsilon}{2} \|\d Z\|_{\dbH^\b(\QQ^\lambda)}.
  \end{align*}
 Choosing $\varepsilon:=\frac{1}{C'_{\beta,L,T}C'_\beta}$, we obtain by Jensen's inequality for $\beta\in(0,1)$ that
  \begin{align*}
     \|\d Z\|_{\dbH^\b(\QQ^\lambda)}
       \leq C'_{\beta,L,T}\big(2+(C'_\beta)^2C'_{\beta,L,T}\big) \|\d Y\|_{\dbS^\b(\QQ^\lambda)} + 2C'_{\beta,L,T}\EE^{\QQ^\lambda}\left[\left(\int_0^T\big|f^0_s\big|\Ind_{\{|f^0_s|\geq n\}}ds\right)\right]^\beta.
  \end{align*}
 Therefore, 
  \begin{align} \label{deltaZleqdeltaY+deltaf}
    & \|\delta Z\|_{\cH^\b} 
    \leq C^Z_{\beta,L,T}\Bigg(\|\delta Y\|_{\cS^\b} + \cE^{\PP}\left[\left(\int_0^T\big|f^0_s\big|\Ind_{\{|f^0_s|\geq n\}}ds\right)\right]^\beta\Bigg).  
  \end{align}
 As the right-hand side converges to $0$ for $m,n\to\infty$, $\{Z^n\}_{n\in\NN}\subseteq\cH^\beta$ is a Cauchy sequence. 
 By completeness of $\cH^\beta$, there exists a limit $Z\in\cH^\beta$. 
 
\ms

\noindent \textbf{Step 3:} We next show that $\{K^n\}_n$ is a Cauchy sequence in $\cS^\beta$. 
By \eqref{deltaRBSDE}, we have
 $$ \delta K_t = -\delta Y_t + \delta Y_0 - \int_0^t\big(g_s(\delta Y_s,\delta Z_s) + \widehat\sigma_s^{\intercal}\delta Z_s\cdot\lambda_s\big)ds + \int_0^t\widehat\sigma_s^{\intercal}\delta Z_s\cdot dW^\lambda_s,  $$
 and therefore, together with Assumption \ref{assum:BSDE.Lipschitz}
 \begin{align*}
  & \sup_{0\leq s\leq T}|\delta K_s| \\
       & \, \leq 2\sup_{0\leq s\leq T}|\delta Y_s| + \int_0^T\left(L_y|\delta Y_s|+2L_z\big|\widehat\sigma_s^{\intercal}\delta Z_s\big| + \big|\delta f_s(Y^m_s,Z^m_s)\big|\right)ds + \sup_{0\leq u\leq T}\left|\int_0^u\widehat\sigma_s^{\intercal}\delta Z_s\cdot dW^\lambda_s\right|  \nonumber \\
       & \, \leq (2+TL_y)\sup_{0\leq s\leq T}|\delta Y_s| + \int_0^T \big( 2L_z \big|\widehat\sigma_s^{\intercal}\delta Z_s\big|ds + \big|f^0_s\big|\Ind_{\{|f^0_s|\geq n\}}\big)ds + \sup_{0\leq u\leq T}\left|\int_0^u\widehat\sigma_s^{\intercal}\delta Z_s\cdot dW^\lambda_s\right|.
 \end{align*}
Applying Cauchy-Schwarz inequality and Burkholder-Davis-Gundy inequality, we obtain that for $\beta\in(0,1)$
 \begin{align*}
\|\d K\|_{\dbS^\b(\QQ^\lambda)}
         &\leq (2+TL_y)^\beta\|\d Y\|_{\dbS^\b(\QQ^\lambda)}
            + \big((2L_z)^\beta T^{\frac{\beta}{2}}+C'_\b\big)\|\d Z\|_{\dbH^\b(\QQ^\lambda)} \\
         & \hspace{42.5mm}  + \EE^{\QQ^\lambda}\left[\left(\int_0^T\big|f^0_s\big|\Ind_{\{|f^0_s|\geq n\}}ds\right)^\beta\right]      \nonumber \\
         & \leq C^K_{\beta,L,T}\Bigg(\|\d Y\|_{\dbS^\b(\QQ^\lambda)} + \|\d Z\|_{\dbH^\b(\QQ^\lambda)} + \EE^{\QQ^\lambda}\Bigg[\left(\int_0^T\big|f^0_s\big|\Ind_{\{|f^0_s|\geq n\}}ds\right)^\beta\Bigg] \Bigg),
 \end{align*}
 and hence
 \begin{align}  \label{deltaLleqdeltaY+deltaZ+deltaf}
\|\d K\|_{\cS^\b}
      & \leq C^K_{\beta,L,T}\Bigg(\|\d Y\|_{\cS^\b} + \|\delta Z\|_{\cH^\b} + \cE^{\PP}\left[\left(\int_0^T\big|f^0_s\big|\Ind_{\{|f^0_s|\geq n\}}ds\right)^\beta\right] \Bigg).
 \end{align}
Since the right-hand side converges to $0$, we obtain 
 $ \lim_{m,n\to\infty}\cE^{\PP}\left[\sup_{0\leq s\leq T}|K^m_s-K^n_s|^\beta\right] = 0, $
 and that by completeness of $\cS^\beta$ there exists a limit $K\in\cS^\beta$.
Clearly, the process $K$ is nondecreasing and predictable. 

\ms 

\noindent \textbf{Step 4:} We now show that the limiting process $(Y,Z,K)$ solves the RBSDE \eqref{eq:RBSDE}.
Indeed, for $t\in[0,T]$ and $\varepsilon>0$, by Markov's inequality and Burkholder-Davis-Gundy inequality, we have
 \begin{align*}
    \PP\left[\sup_{0\leq u\leq t}\left|\int_0^uZ^n_s\cdot dB_s-\int_0^u Z_s\cdot dB_s\right|>\varepsilon\right] 
        &\leq \frac{1}{\varepsilon^\beta}\EE\left[\sup_{0\leq u\leq t}\left|\int_0^u\widehat\sigma_s^{\intercal}(Z^n_s-Z_s) \cdot dW_s\right|^\beta\right] \\  
        &\leq \frac{C_\beta}{\varepsilon^\beta}\|Z^n-Z\|_{\dbH^\b}, 
 \end{align*}
 which implies that $\int_0^T Z^n_s\cdot dB_s$ converges to $\int_0^tZ_s\cdot dB_s$ in ucp (uniform convergence on compacts in probability).  
We may extract a subsequence such that $\int_0^TZ_s^n\cdot dB_s$ converges to $\int_0^TZ_s\cdot dB_s$ almost surely.
By Assumption \ref{assum:BSDE.Lipschitz}, we have 
 \begin{align*}
  \int_0^T & \big|f^n_s(Y_s^n,Z_s^n)-f_s(Y_s,Z_s)\big|ds \\
     & \leq \int_0^T\big|f^0_s\big|\Ind_{\{|f^0_s|\geq n\}}ds + L_y\int_0^T|Y^n_s-Y_s|ds + L_z\int_0^T\big|\widehat\sigma_s^{\intercal}(Z^n_s-Z_s)\big|ds \\
     & \leq \int_0^T\big|f^0_s\big|\Ind_{\{|f^0_s|\geq n\}}ds + L_yT\sup_{0\leq s\leq T}|Y^n_s-Y_s| + L_zT^{\frac{1}{2}}\left(\int_0^T\big|\widehat\sigma_s^{\intercal}(Z^n_s-Z_s)\big|^2ds\right)^{\frac{1}{2}} ~\rightarrow ~ 0, 
 \end{align*}
 as $n\to\infty$, and hence $(Y,Z,K)$ solves the correct RBSDE.

\ms

\noindent \textbf{Step 5:} 
We now apply Snell envelop approach to optimal stopping in order to show the Skorokhod condition.
Define 
 $$ \tau_m:= \left\{t>0\,:\, \sup_{n\in\NN}\int_0^t\big|f^n_s(Y^n_s,Z^n_s)\big|ds + \sup_{n\in\NN}\left|\int_0^tZ^n_s\cdot dB_s\right| \geq m \right\}. $$
Up to time $\tau_m$, It\^o integrals are true martingales which disappear after taking conditional expectation, and Lebesgue integals are uniformly bounded which allows us to apply the dominated convergence theorem easily. 
By Step 4 we have  
  $$ \sup_{n\in\NN}\sup_{0\leq u\leq T}\left|\int_0^u Z^n_s\cdot dB_s\right| <\infty, \quad \mbox{and} \quad \sup_{n\in\NN}\sup_{0\leq u\leq T} \int_0^u\big|f^n_s(Y_s^n,Z_s^n)\big|ds < \infty, \quad a.s., $$
 which implies that the sequence of stopping times $\{\tau_m\}_{m\in\NN}$ converges almost surely to $\infty$.
By following the proof of \cite[Proposition 3.1]{LX05}, we may show the following representation for each $n\in\NN$ 
  \begin{align*}
   Y^n_{t\wedge\tau_m} 
        &= \esssup_{\tau\in\cT_{t\wedge\tau_m,T\wedge\tau_m}} \EE\left[\int_{t\wedge\tau_m}^{\tau}f_s^n(Y_s^n,Z_s^n)ds + S_\tau\Ind_{\{\tau<T\wedge\tau_m\}}+Y^n_{T\wedge\tau_m}\Ind_{\{\tau=T\wedge\tau_m\}} \bigg|\cF_{t\wedge\tau_m}^{+,\PP}\right].
  \end{align*}
By triangle inequality, we have 
  \begin{align*}
   Y^n_{t\wedge\tau_m} 
        &\leq \esssup_{\tau\in\cT_{t\wedge\tau_m,T\wedge\tau_m}} \EE\left[\int_{t\wedge\tau_m}^{\tau}f_s(Y_s,Z_s)ds + S_\tau\Ind_{\{\tau<T\wedge\tau_m\}}+Y_{T\wedge\tau_m}\Ind_{\{\tau=T\wedge\tau_m\}} \bigg|\cF_{t\wedge\tau_m}^{+,\PP}\right] \\
        & \hspace{8mm} + \EE\left[\int_0^{T\wedge\tau_m}\big|f^n_s(Y^n_s,Z^n_s)-f_s(Y_s,Z_s)\big|ds + \big|Y^n_{T\wedge\tau_m} - Y_{T\wedge\tau_m}\big| \bigg|\cF_{t\wedge\tau_m}^{+,\PP}\right].
  \end{align*}
It follows from $Y^n_{t\wedge\tau_m} \to Y_{t\wedge\tau_m}$, $Y^n_{T\wedge\tau_m} \to Y_{T\wedge\tau_m}$ and $\int_0^{T\wedge\tau_m}\big|f^n_s(Y^n_s,Z^n_s)-f_s(Y_s,Z_s)\big|ds\to 0$ in $\LL^1$ that 
 \begin{align*} 
   Y_{t\wedge\tau_m}
    \leq \esssup_{\tau\in\cT_{t\wedge\tau_m,T\wedge\tau_m}}\EE\left[\int_{t\wedge\tau_m}^{\tau}f_s(Y_s,Z_s)ds + S_\tau\Ind_{\{\tau<T\wedge\tau_m\}} + Y_{T\wedge\tau_m}\Ind_{\{\tau=T\wedge\tau_m\}} \bigg|\cF_{t\wedge\tau_m}^{+,\PP}\right]. 
 \end{align*}
On the other hand, it is clear that $ Y_{t\wedge\tau_m} \geq S_{t\wedge\tau_m}\Ind_{\{t< T\wedge\tau_m\}} + Y_{T\wedge\tau_m}\Ind_{\{t=T\wedge\tau_m\}}. $
Since $Y_{t\wedge\tau_m} + \int_0^{t\wedge\tau_m}f_s(Y_s,Z_s)ds$ is a supermartingale, we have 
 \begin{align*}
   & Y_{t\wedge\tau_m} + \int_0^{t\wedge\tau_m}f_s(Y_s,Z_s)ds \\
   & \quad \geq \esssup_{\tau\in\cT_{t\wedge\tau_m,T\wedge\tau_m}}\EE\left[\int_{0}^{\tau}f_s(Y_s,Z_s)ds + S_\tau\Ind_{\{\tau<T\wedge\tau_m\}} + Y_{T\wedge\tau_m}\Ind_{\{\tau=T\wedge\tau_m\}} \bigg|\cF_{t\wedge\tau_m}^{+,\PP}\right], 
 \end{align*}
 and therefore, 
 \begin{align} \label{eq:snellY1}
   Y_{t\wedge\tau_m} = \esssup_{\tau\in\cT_{t\wedge\tau_m,T\wedge\tau_m}} \EE\left[\int_{t\wedge\tau_m}^{\tau}f_s(Y_s,Z_s)ds + S_\tau\Ind_{\{\tau<T\wedge\tau_m\}} + Y_{T\wedge\tau_m}\Ind_{\{\tau=T\wedge\tau_m\}} \bigg|\cF_{t\wedge\tau_m}^{+,\PP}\right].
 \end{align}
Define 
 \begin{align*}
  \eta^m_t:= \int_0^tf_s(Y_s,Z_s)ds &+ S_t\Ind_{\{t<T\wedge\tau_m\}} + Y_{T\wedge\tau_m}\Ind_{\{t=T\wedge\tau_m\}} \\
                                    &- \EE\left[\int_{0}^{T\wedge\tau_m}f_s(Y_s,Z_s)ds + Y_{T\wedge\tau_m}\bigg|\cF_{t}^{+,\PP}\right].
 \end{align*} 
Clearly, $\eta^m_{T\wedge\tau_m}=0$. Note that $\eta^m=(\eta^m_t)_{0\leq t\leq T\wedge\tau_m}$ is of class $\DD(\PP)$. 
Let $J^m_t$ be the Snell envelope of $\eta^m$  
  $$ J^m_t := \esssup_{\tau\in\cT_{t\wedge\tau_m,T\wedge\tau_m}}\EE\Big[\eta^m_\tau\Big|\cF_{t\wedge\tau_m}^{+,\PP}\Big],  $$
  which is a c\`adl\`ag process of class $\DD(\PP)$ and is the smallest supermartingale dominating the process $\eta$. 
Hence, by the Doob-Meyer decomposition, there exist a martingale $M^m$ and a predictable nondecreasing process $A^m$ such that $J^m_t = J^m_0 + M^m_t - A^m_t$.  
By the definition of $J^m$ and the representation \eqref{eq:snellY1}, we obtain 
\begin{align}
 J^m_t = Y_{t\wedge\tau_m} - \EE\left[\int_{t\wedge\tau_m}^{T\wedge\tau_m}f_s(Y_s,Z_s)ds + Y_{T\wedge\tau_m}\bigg|\cF_{t\wedge\tau_m}^{+,\PP}\right].
\end{align} 
We have that 
 $$ J^m_t + \EE\left[\int_{0}^{T\wedge\tau_m}f_s(Y_s,Z_s)ds + Y_{T\wedge\tau_m}\bigg|\cF_{t\wedge\tau_m}^{+,\PP}\right] = Y_{t\wedge\tau_m} + \int_{0}^{t\wedge\tau_m}f_s(Y_s,Z_s)ds $$
 is a supermartingale.
Therefore, by the uniqueness of the Doob-Meyer decomposition, $A^m_t = K_{t\wedge\tau_m}$. 
Decompose $A^m$ (and the same for $K$) in continuous part $A^{m,c}$ ($K^c$) and pure-jumps part $A^{m,d}$ ($K^d$). 
By \cite[Proposition B.11]{KQ12}, see also \cite[Proposition 2.34]{ElK81}, we obtain 
 $$ \int_0^{T\wedge\tau_m}\big(J^m_t - \eta^m_t\big)dA_t^{m,c} = 0 \quad \mbox{and} \quad \Delta A^{m,d}_t = \Delta A^{m,d}_t\Ind_{\{J^m_{t-}=\eta^m_{t-}\}}, \,\, a.s. \,\mbox{ for } t\leq\tau_m. $$
By noticing that $ J^m_t - \eta^m_t = Y_{t\wedge\tau_m} - S_{t\wedge\tau_m}\Ind_{\{t<T\wedge\tau_m\}} - Y_{T\wedge\tau_m}\Ind_{\{t=T\wedge\tau_m\}}$,
 we obtain that
 $$ \int_0^{T\wedge\tau_m}\big(Y_{t-} - S_{t-}\big)dK_t = \int_0^{T\wedge\tau_m}\big(J^m_{t-} - \eta^m_{t-}\big)dK_t = 0,\quad a.s. $$
Letting $m\to\infty$, the Skorokhod condition holds true for $K$.


 
 

%
%
 

\ms
 
\noindent \textbf{Uniqueness:}  
 Let $(Y,Z,K)$ and $(Y',Z',K')$ be two solutions to RBSDE$(f,\xi, S)$. 
 Set $\delta Y = Y-Y'$, $\delta Z=Z-Z'$ and $\delta K=K-K'$. 
 Using the similar computation as above, we have  
  \begin{align*}
   |\delta Y_{\tau\wedge \tau_m}| \leq |\delta Y_{\tau_m}| & - \int_{\tau\wedge\tau_m}^{\tau_m}\sgn(\delta Y_{s})\widehat\sigma_s^{\intercal}\delta Z_s\cdot \left(dW_s - \widehat\lambda_s ds\right), \nonumber        
  \end{align*}
  where $ \widehat\lambda_s:= L\sgn(\delta Y_s)\frac{\widehat\sigma_s^{\intercal}\delta Z_s}{\big|\widehat\sigma_s^{\intercal}\delta Z_s\big|}\Ind_{\{|\widehat\sigma_s^{\intercal}\delta Z_s|\neq 0\}} $
   and $$ \tau_m:=\inf\bigg\{t\geq 0\, :\, \int_0^t |Z_s|^2+|Z'_s|^2 ds\geq m\bigg\}\wedge T. $$
 Taking the conditional expectation with respect to $\cF_\tau$ under the equivalent measure $\widehat\QQ\sim\PP$, defined by $\frac{d\widehat\QQ}{d\PP}=G^{\widehat\lambda}_T$,
  we obtain that
   $$ |\delta Y_{\tau\wedge\tau_m}| \leq \EE^{\widehat{\QQ}}\big[|\delta Y_{\tau_m}|\big|\cF_\tau^{+,\PP}\big]. $$
 Again, since $\delta Y$ belongs to $\DD(\widehat{\QQ})$, it follows that $\delta Y_{\tau_m}\to 0$ in $\LL^1(\widehat\QQ)$, therefore
   $|\delta Y_{\tau}|=0$. 
 It follows by the section theorem that $Y$ and $Y'$ are indistinguishable.  
 By \eqref{deltaZleqdeltaY+deltaf} and \eqref{deltaLleqdeltaY+deltaZ+deltaf}, $(\delta Z,\delta K)=(0,0)$. 
\end{proof}

\subsubsection{General obstacles}

Before proving the wellposedness result, we state the following comparison result for the general c\`adl\`ag solution in the $\LL^2$-setting, which is a generalization of \cite[Proposition 3.2]{RS12}.  
The proof is omitted as it follows the same argument as in the classical one.

\begin{proposition}  \label{comparison1}
 Let $(f,\xi,S)$ and $(f',\xi',S')$ be such that $f$ and $f'$ satisfy Assumption \ref{assum:BSDE.Lipschitz} and 
 $ \EE\big[\int_0^T\big|f^0_s\big|^2ds\big]<\infty$, $\EE\big[\int_0^T\big|f'^0_s\big|^2ds\big]<\infty,$ $\xi,\xi'\in\LL^2$, and $S,S'\in\SSS^2$, and let $(Y,Z,K)$ and $(Y',Z',K')$ be the  corresponding solutions.
 
 Assume that $\xi\le\xi'$, $S_t\leq S'_t$, and $f_t(Y'_t,Z'_t)\leq f'_t(Y'_t,Z'_t)$, $\PP$-a.s., $t\in[0,T]$. 
 Then $Y_\tau\leq Y'_\tau$, for all $\tau\in\cT_{0,T}$.
\end{proposition}

\begin{proposition} \label{comparison2}
 Let $(f,\xi,S)$ and $(f',\xi',S')$ satisfy the assumptions of Theorem \ref{ExistenceUniqueness112}. 
 Let $(Y,Z,K)$ and $(Y', Z', K')$ be solutions of corresponding RBSDEs.
 Suppose that $\xi\leq  \xi'$, $\PP$-a.s.; $f(y,z)\leq f'(y,z)$, $dt\otimes\PP$-a.e., for each $y,z\in\RR\times\RR^d$; and $S\leq S'$, $dt\otimes\PP$-a.e.
 Then, $Y_\tau\leq Y'_\tau$ for each $\tau\in\cT_{0,T}$. 
\end{proposition}

\begin{proof}
 Let $(Y^n,Z^n,K^n)$ and $(Y'^n,Z'^n,K'^n)$ be the approximation sequences of the solutions of RBSDE with $(f,\xi,S)$ and $(f',\xi',S')$, respectively. 
 By the comparison result, Proposition \ref{comparison1}, we have $Y_\tau^n\leq Y'^n_\tau$, therefore $Y_\tau\leq Y'_\tau$ for each $\tau\in\cT_{0,T}$.
\end{proof}

Now we are ready to prove the main theorem.

\begin{proof}[Proof of Theorem \ref{RBSDEmainthm}]
Define $S_t^n:= S_t\wedge n$. Clearly, $S^n\geq S^m$ for $n\geq m$. By Theorem \ref{ExistenceUniqueness112}, RBSDE with $(f,\xi,S^n)$ has a unique solution $(Y^n,Z^n,K^n)$.
Define $(\delta Y,\delta Z,\delta K):=(Y^n-Y^m, Z^n-Z^m, K^n-K^m)$. 
By Proposition \ref{comparison2} we have $\delta Y\geq 0$. 
\ms

\noindent\textbf{Step 1:} We are going to show that $\{Y^n\}_{n\in\NN}$ is a Cauchy sequence in $\cD$ and $\cS^{\beta}$. 
Let $\sigma\in\cT_{0,T}$ be arbitrary. 
Define 
 \begin{align*}
   & \tau_\sigma^\varepsilon:=\inf\{t\geq \sigma \,:\, Y_t^n \leq S_t^n + \varepsilon\}\wedge T, \\
   & \tau_k:=\inf\left\{t\geq 0 \,:\, \int_0^t\big(\big|\widehat\sigma_s^{\intercal}Z_s^m\big|^2+\big|\widehat\sigma_s^{\intercal}Z_s^n\big|^2\big)ds\geq k \right\}\wedge\tau^N_k\wedge T. 
 \end{align*}
It follows from the definition of $\tau_\sigma^\varepsilon$ that $K^n$ is flat on $\llbracket \sigma,\tau_\sigma^\varepsilon \rrbracket$, hence 
 \begin{equation*}
   \sgn(\delta Y_{s-})d\delta K_s = \sgn(Y^n_{s-}-Y^m_{s-})dK^n_s - \sgn(Y^n_{s-}-Y^m_{s-})dK^m_s \leq 0 
 \end{equation*} 
 on $\llbracket \sigma,\tau_\sigma^\varepsilon \rrbracket$. 
Again by Proposition 4.2 in \cite{LRTY20} and Assumption \ref{assum:BSDE.Lipschitz}, we obtain
  \allowdisplaybreaks
  \begin{align*}
   |\delta Y_{\sigma\wedge\tau_k}| & \leq |\delta Y_{\tau_\sigma^\varepsilon\wedge\tau_k}| - \int_{\sigma\wedge\tau_k}^{\tau_\sigma^\varepsilon\wedge\tau_k}\sgn(\delta Y_{s-})d\delta Y_s \\
     & = |\delta Y_{\tau^\varepsilon_\sigma\wedge\tau_k}| + \int_{\sigma\wedge\tau_k}^{\tau_\sigma^\varepsilon\wedge\tau_k}\sgn(\delta Y_{s})\big(f_s(Y^n_s,Z^n_s)-f_s(Y^m_s, Z^m_s)\big)ds \\
     & \quad - \int_{\sigma\wedge\tau_k}^{\tau_\sigma^\varepsilon\wedge\tau_k}\sgn(\delta Y_{s})\widehat\sigma_s^{\intercal}\delta Z_s\cdot dW_s + \int_{\sigma\wedge\tau_k}^{\tau_\sigma^\varepsilon\wedge\tau_k}\sgn(\delta Y_{s-})d\delta K_s \\
     & \leq |\delta Y_{\tau_\sigma^\varepsilon\wedge\tau_k}| - \int_{\sigma\wedge\tau_k}^{\tau_\sigma^\varepsilon\wedge\tau_k} \sgn(\delta Y_{s})\widehat\sigma_s^{\intercal}\delta Z_s\cdot
        \left(dW_s - \widehat\lambda_s ds\right), 
  \end{align*}
 where $ \widehat\lambda_s:= L_z\sgn(\delta Y_s)\frac{\widehat\sigma_s^{\intercal}\delta Z_s}{|\widehat\sigma_s^{\intercal}\delta Z_s|}\Ind_{\{|\widehat\sigma_s^{\intercal}\delta Z_s|\neq 0\}}$.
Conditioning with respect to $\cF_{\sigma}^{\PP}$ under the equivalent measure $\widehat\QQ\in\cQ_{L_z}$ defined by $ \frac{d\widehat{\QQ}}{d\PP} = G^{\widehat\lambda}_T$,
 and then, as $\delta Y$ is of class $\cD$, letting $k\to\infty$, we deduce from the above inequality that
 \begin{equation*}
   |\delta Y_{\sigma}| \leq \EE^{\widehat\QQ}\left[\big|\delta Y_{\tau_\sigma^\varepsilon}\big|\Big|\cF^{+,\PP}_\sigma\right]
   \le \EE^{\widehat\QQ}\left[S^+_{\tau_\sigma^\varepsilon}\Ind_{\big\{S^+_{\tau_\sigma^\varepsilon}\geq m\big\}}\bigg|\cF^{+,\PP}_\sigma\right] + \varepsilon,
 \end{equation*}
 where the last inequality follows from  
   $$ 0\leq \delta Y_{\tau_\sigma^\varepsilon} = S_{\tau_\sigma^\varepsilon}^n + \varepsilon - Y^m_{\tau_\sigma^\varepsilon} \leq S_{\tau_\sigma^\varepsilon}^n + \varepsilon - S_{\tau_\sigma^\varepsilon}^m 
        \leq S^+_{\tau_\sigma^\varepsilon}\Ind_{\big\{S^+_{\tau_\sigma^\varepsilon}\geq m\big\}} + \varepsilon. $$
Let $\QQ\in\cQ_{L_z}$ be arbitrary. We obtain that 
 \begin{align*}
  \EE^{\QQ}[|\delta Y_{\sigma}|] &\leq \EE^{\QQ}\left[\EE^{\widehat\QQ}\left[S^+_{\tau_\sigma^\varepsilon}\Ind_{\big\{S^+_{\tau_\sigma^\varepsilon}\geq m\big\}}\bigg|\cF^{+,\PP}_\sigma\right]\right] + \varepsilon  
           \leq \sup_{\tau\in\cT_{0,T}}\cE^{\PP}\left[S^+_\tau\Ind_{\big\{S^+_{\tau}\geq m\big\}}\right] + \varepsilon.
 \end{align*}
Since the above inequality holds for all $\varepsilon>0$, $\sigma\in\cT_{0,T}$ and $\QQ\in\cQ_{L_z}$, we obtain that 
 $$ \sup_{\sigma\in\cT_{0,T}}\cE^{\PP}[|\delta Y_{\sigma}|] \leq \sup_{\tau\in\cT_{0,T}}\cE^{\PP}\left[S^+_{\tau}\Ind_{\{S^+_{\tau}\geq m\}}\right]. $$

\no Together with Lemma \ref{ELbetaleqsupEL}, we obtain that for any $\beta\in(0,1)$
 \begin{equation*}
  \begin{aligned}
  \|\d Y\|_{\cS^\b} &\leq \frac{1}{1-\beta}\|\d Y\|_{\cD}^\beta 
      \leq \frac{1}{1-\beta}\sup_{\tau\in\cT_{0,T}}\cE^{\PP}\left[S^+_{\tau}\Ind_{\{S^+_{\tau}\geq m\}}\right]^\beta.    
  \end{aligned}
 \end{equation*}
 As the spaces $\cD$ and $\cS^\beta$ are complete, we may find a limit $Y\in\cD\cap\cS^\beta$.
 
\ms

\noindent \textbf{Step 2:} We will show that $\{Z^n\}_{n\in\NN}$ is a Cauchy sequence in $\cH^\beta$.
Similar to \eqref{itosameS}, we have 
  \allowdisplaybreaks
 \begin{align}  \label{deltaY^2=Ito}
  \int_0^T  \big|\widehat\sigma_s^{\intercal}\delta Z_s\big|^2ds \leq 2(2+8L_z^2T)\sup_{0\leq s\leq T}(\delta Y_s)^2 &+ 4\int_0^T\delta Y_{s-}d\delta K_s \nonumber \\
    & + 4\sup_{0\leq u\leq T}\left|\int_0^u\delta Y_{s}\widehat\sigma_s^{\intercal}\delta Z_s\cdot dW^\lambda_s\right|.
 \end{align} 
Comparing to \eqref{itosameS}, the extra term $\d Y_{s-} d\d K_s$ is due to the different obstacles $S^n\neq S^m$. Note that by Skorokhod condition and $Y^m\geq S^m$, $Y^n\geq S^n$
  \allowdisplaybreaks
   \begin{align*}
     \int_0^t\delta Y_{s-}d\delta K_s & = \int_0^t(Y^n_{s-}-S^n_{s-})dK^n_s + \int_0^t(S^n_{s-}-Y^m_{s-})dK^n_s  \\
       & \quad -\int_0^t(Y^n_{s-}-S^m_{s-})dK^m_s + \int_0^t(Y^m_{s-}-S^m_{s-})dK^m_s  \\
       & = \int_0^t S^n_{s-}dK_s^n - \int_0^tY_{s-}^mdK_s^n -\int_0^tY_{s-}^ndK_s^m + \int_0^tS_{s-}^mdK_s^m \\
       & \leq \int_0^t \delta S_{s-}dK_s^n -\int_0^t\delta S_{s-}dK_s^m \\
       & \leq \sup_{0\leq s\leq T}|\delta S_s|(K^m_T+K^n_T). 
   \end{align*}  
 Plugging this inequality in \eqref{deltaY^2=Ito}, we obtain 
 \begin{align}   
  \int_0^T \big|\widehat\sigma_s^{\intercal}\delta Z_s\big|^2ds \leq 2(2+8L_z^2T)\sup_{0\leq s\leq T}(\delta Y_s)^2 
         &+ 4\sup_{0\leq s\leq T}|\delta S_s|(K^m_T+K^n_T)  \nonumber \\
         &+ 4\sup_{0\leq u\leq T}\left|\int_0^u\delta Y_{s}\widehat\sigma_s^{\intercal}\delta Z_s\cdot dW^\lambda_s\right|.
 \end{align}
Taking expectation and using Young and Cauchy-Schwarz inequality, we obtain that 
\begin{align*}
  \|\d Z\|_{\dbH^\b(\QQ^\lambda)} 
        \leq C_{\beta,L,T}\Bigg( \|\d Y\|_{\dbS^\b(\QQ^\lambda)} 
           +\|\d S\|_{\dbS^\b(\QQ^\lambda)}^{\frac{1}{2}}\bigg(\|K^m\|_{\II^\beta(\QQ^\lambda)}^{\frac{1}{2}} + \|K^n\|_{\II^\beta(\QQ^\lambda)}^{\frac{1}{2}}\bigg)\Bigg).
\end{align*}
 Since $0\leq \delta S_t\leq S^{+}_t\Ind_{\{S^{+}_t\geq m\}}$ for $t\in[0,T]$, we have $ 0\leq \sup_{0\leq s\leq T}|\delta S_s|\leq \sup_{0\leq s\leq T}S_s^{+}\Ind_{\{S_s^{+}\geq m\}}$ 
   and by Lemma \ref{ELbetaleqsupEL} 
   \begin{equation*}
   \|\d S\|_{\dbS^\b(\QQ^\lambda)}
   \leq \EE^{\QQ^\lambda}\left[\left(\sup_{0\leq s\leq T}S_s^{+}\Ind_{\{S_s^{+}\geq m\}}\right)^\beta\right] 
         \leq \frac{1}{1-\beta}\sup_{\tau\in\cT_{0,T}}\left(\EE^{\QQ^\lambda}\left[S^+_\tau\Ind_{\{S^+_\tau\geq m\}}\right]\right)^\beta.
   \end{equation*}
 As in the proof of Theorem \ref{ExistenceUniqueness112}, by the convergence of $\{Y^n\}_{n\in\NN}$ and Proposition \ref{EstimationZNK} we have $\sup_{n\in\NN}\|K^n\|_{\cS^\beta}<\infty$.
 Therefore, we obtain
  \begin{align*}
    \hspace{-2mm}\|\d Z\|_{\cH^\b}\leq C_{\beta,L,T,Y}\Bigg(\|\d Y\|_{\cS^\b} + \sup_{\tau\in\cT_{0,T}}\cE^{\PP}\left[S_\tau^+\Ind_{\{S^+_\tau\geq m\}}\right]^{\frac{\beta}{2}}\Bigg).
  \end{align*}
 As the right-hand side converges to $0$, $\{Z^n\}_{n\in\NN}\subseteq\cH^\beta$ is a Cauchy sequence. 
 Again by completeness, there exists a limit $Z\in\cH^\beta$. 
 
\ms

\noindent \textbf{Step 3:} Using the same argument as in {\bf Step 3} of the proof of Theorem \ref{ExistenceUniqueness112}, we show 
 that $\{K^n\}_{n\in\NN}$ is a Cauchy sequence in $\cS^\beta$. 
 Hence, there exists a limit $K\in\cS^\beta$.
 It is clear that $K$ is a nondecreasing predictable process starting from zero.
 
\ms 
 Clearly $Y\geq S$. 
 In the same way, we show the Skorokhod condition and that $(Y,Z,K)$ solves the correct RBSDE with $(f,\xi,S)$.  
 
\ms 
 The uniqueness follows by the same argument as in the proof of Theorem \ref{ExistenceUniqueness112}.
\end{proof}

\begin{proof}[Proof of Theorem \ref{thm:RBSDE.stability.comparison} $(ii)$]
 We denote by $(Y^n,Z^n,K^n)$ and $(Y'^n,Z'^n,K'^n)$ the approximation sequences of the solutions of RBSDEs with $(f,\xi,S)$ and $(f',\xi',S')$, respectively. 
 By the comparison result, Proposition \ref{comparison2}, we have $Y_\tau^n\leq Y'^n_\tau$, therefore $Y_\tau\leq Y'_\tau$ for each $\tau$.
\end{proof}

\subsection{Stability of reflected BSDE}

\begin{proof}[Proof of Theorem \ref{thm:RBSDE.stability.comparison} $(i)$]
 Obviously, the process $(\delta Y,\delta Z,\delta K)$ satisfies 
  $$ 
  \delta Y_t 
  = 
  \delta\xi + \int_t^Tg_s(\delta Y_s,\delta Z_s)ds - \int_t^T\delta Z_s\cdot dX_s + \int_t^Td\delta K_s, 
  \quad \PP\mbox{-a.s.}
  $$
where $g_s(\delta Y_s,\delta Z_s):=f'_s(Y_s+\delta Y_s,Z_s+\delta Z_s)-f_s(Y_s,Z_s)$. 
 Define $ \tau_m:=\inf\{t\geq 0 : \int_0^t |Z_s|^2+|Z'_s|^2 ds\geq m\}\wedge T$. 
 Following the same argument as in {\bf Step 1} of the proof of Theorem \ref{ExistenceUniqueness112}, we obtain that 
  \begin{align*}
    |\delta Y_{\tau\wedge\tau_m}| 
       & \leq 
       |\delta Y_{T\wedge\tau_m}| 
       + \int_0^{T}|\delta f_s(Y_s,Z_s)|ds 
       - \int_{\tau\wedge\tau_m}^{T\wedge\tau_m}
                \sgn(\delta Y_{s-})\widehat\sigma_s^{\intercal}\delta Z_s\cdot \big(dW_s - \widehat\lambda_s ds\big), 
  \end{align*}
   where $ \widehat\lambda_s:= L_z\sgn(\delta Y_s)\frac{\widehat{\sigma}_s^{\intercal}\delta Z_s}{\big|\widehat{\sigma}_s^{\intercal}\delta Z_s\big|}\Ind_{\{|\widehat{\sigma}_s^{\intercal}\delta Z_s|\neq 0\}}$. 
   Define  $d\QQ^{\widehat\lambda} := G^{\widehat\lambda}_T d\PP$.
 Since $\delta Y\in\cD$, it follows that $\delta Y_{\tau_m}\to\delta\xi$ in $\cL^1$, and therefore
    $$ |\delta Y_\tau| \leq \EE^{\QQ^{\widehat\lambda}}\left[|\delta\xi| + \int_0^T|\delta f_s(Y_s,Z_s)|ds\bigg|\cF_\tau^{+,\PP}\right]. $$
 We deduce immediately from Lemma \ref{techLemma1} that 
  \begin{equation}
    \|\d Y\|_{\cD} \leq \cE^{\PP}\left[|\delta\xi| + \int_0^T|\delta f_s(Y_s,Z_s)|ds\right], 
  \end{equation}
 and from Lemma \ref{supDoob} that for any $\beta\in(0,1)$, 
  \begin{equation} \label{eq:RBSDEstab.Y.beta}
  \|\d Y\|_{\cS^\b} \leq \frac{1}{1-\beta}\cE^{\PP}\left[|\delta\xi| + \int_0^T|\delta f_s(Y_s,Z_s)|ds\right]^\beta.
  \end{equation}
 Further, following {\bf Step 2} and {\bf Step 3} in Theorem \ref{ExistenceUniqueness112}, we obtain that 
  \begin{align*}
     \|\d Z\|_{\cH^\b} + \|\delta K\|_{\cS^\beta} \leq C^1_{\beta,L,T}\Bigg(\|\d Y\|_{\cS^\b}+ \cE^{\PP}\left[\left(\int_0^T|\delta f_s(Y_s,Z_s)|ds\right)^\beta\right]\Bigg). 
  \end{align*}
 The assertion follows from \eqref{eq:RBSDEstab.Y.beta}.
\end{proof}

\subsection{A priori estimates and stability of BSDE}

For $S=-\infty$, we have the existence and uniqueness result of the BSDE.
As we have seen in Proposition \ref{EstimationZNK}, there is no a priori estimate for $Y$ for reflected BSDE. 
However, for the BSDE \eqref{eq:BSDE} without reflection we may find a priori estimate for $Y$.

\begin{proof}[Proof of Theorem \ref{thm:ExistenceUniquenessBSDE}: estimates \eqref{BSDEEstimationEq1}-\eqref{BSDEEstimationEq2}]
 Let
  $ \tau_n:=\inf\{t\geq 0:\int_0^t |\widehat\sigma_s^{\intercal}Z_s|^2ds\geq n\}\wedge T$.
 Applying Tanaka's formula, by Assumption \ref{assum:BSDE.Lipschitz} and Remark \ref{rem:monotonicity}, we obtain that
  \begin{align*}
    |Y_{\tau\wedge\tau_n}| &\leq |Y_{\tau_n}| - \int_{\tau\wedge\tau_n}^{\tau_n}\sgn(Y_{s-})dY_s \\
      & = |Y_{\tau_n}| + \int_{\tau\wedge\tau_n}^{\tau_n}\sgn(Y_{s})f_s(Y_s,Z_s)ds - \int_{\tau\wedge\tau_n}^{\tau_n}\sgn(Y_{s})Z_s\cdot dB_s \\
      & \leq |Y_{\tau_n}| + \int_0^T\big|f^0_s \big|ds - \int_{\tau\wedge\tau_n}^{\tau_n}\sgn(Y_{s})\widehat\sigma^{\intercal}_sZ_s\cdot dW_s^{\widehat\lambda},
  \end{align*}
  with $\widehat\lambda_s:= L_z\sgn(Y_s)\frac{\widehat\sigma^{\intercal}_sZ_s}{|\widehat\sigma^{\intercal}_sZ_s|}\Ind_{\{|\widehat\sigma^{\intercal}_sZ_s|\neq 0\}}$.
 Taking conditional expectation with respect to $\cF_\tau^{+,\PP}$ under the measure $\QQ^{\widehat\lambda}$ defined by
  $ \frac{d\QQ^{\widehat\lambda}}{d\PP}:=G^{\widehat\lambda}_T, $
  we obtain 
   $ |Y_{\tau\wedge\tau_n}| \leq \EE^{\QQ^{\widehat\lambda}}\bigg[|Y_{\tau_n}|+\int_0^T\big|f^0_s \big|ds\bigg|\cF_\tau^{+,\PP}\bigg]. $
 As $Y\in\cD$, letting $n\to\infty$, we obtain that 
  \begin{equation}  \label{BSDEEstimationProofeq1}
    |Y_{\tau}| \leq \EE^{\QQ^{\widehat\lambda}}\left[|\xi|+\int_0^T\big|f^0_s \big|ds\bigg|\cF_\tau^{+,\PP}\right],
  \end{equation}
   and \eqref{BSDEEstimationEq1} follows immediately from Lemma \ref{techLemma1}.
 By Lemma \ref{supDoob}, \eqref{BSDEEstimationProofeq1} implies that
   \begin{equation} \label{BSDEEstimationProofeq2}
    \|Y\|_{\cS^\b} \leq \frac{1}{1-\beta}\cE^{\PP}\left[|\xi|+\int_0^T\big|f^0_s \big|ds\right]^\beta
    \q\mbox{for all}\q\b\in (0,1).
   \end{equation}
 Further, by applying It\^o's formula on $Y^2$, we see that
  \allowdisplaybreaks
  \begin{align*}
    \int_0^T \big|\widehat\sigma_s^{\intercal}Z_s\big|^2ds 
        &= Y^2_T-Y^2_0 + 2\int_0^TY_s\big(f_s(Y_s,Z_s)- \widehat\sigma_s^{\intercal}Z_s\cdot \lambda_s\big)ds - 2 \int_0^TY_sZ_s\cdot dB^\lambda_s  \\
        & \leq \sup_{0\leq s\leq T}Y^2_s + 2\int_0^T|Y_s| \big(L_z\big|\widehat\sigma^{\intercal}_sZ_s\big| + \big|f_s^0\big|\big)ds 
        + 2\sup_{0\leq u\leq T}\left|\int_0^uY_sZ_s\cdot dB^\lambda_s\right| \\
        & \leq (2+2L^2_zT)\sup_{0\leq s\leq T}|Y_s|^2 + \frac{1}{2}\int_0^T\big|\widehat\sigma_s^{\intercal}Z_s\big|^2ds + \left(\int_0^T\big|f^0_s\big|ds\right)^2 \\
        & \hspace{74mm} + 2\sup_{0\leq u\leq T}\left|\int_0^uY_sZ_s\cdot dB^\lambda_s\right|.
  \end{align*}
  Finally, by Burkholder-Davis-Gundy inequality and Young's inequality, we obtain  
    \begin{align*}
      \|Z\|_{\HH^\beta(\QQ^\lambda)} \leq C''_{L,T,\beta}\left(\|Y\|_{\SSS^\beta(\QQ^\lambda)} + \EE^{\QQ^\lambda}\left[\left(\int_0^T\big|f^0_s\big|ds\right)^\beta\right]\right).
    \end{align*}
  Taking supremum over all $\QQ^\lambda\in\cQ_{L_z}$,  \eqref{BSDEEstimationEq2} follows from the above inequality and \eqref{BSDEEstimationProofeq2}.
\end{proof}

As in the proof of Theorem \ref{ExistenceUniqueness112} and that of the estimates \eqref{BSDEEstimationEq1}-\eqref{BSDEEstimationEq2} of Theorem \ref{thm:ExistenceUniquenessBSDE}, we may estimate the difference of two solutions of two BSDEs. 
Let $(Y^n,Z^n)$ be the solution of the approximating BSDE with $(f^n,\xi^n)$ as in previous section. Define $(\delta Y,\delta Z):=(Y-Y^n,Z-Z^n)$. 

\begin{proposition}\label{prop:deltaYestimate}
Under Assumptions \ref{assum:BSDE.Lipschitz} and \ref{assum:BSDExi}, we have 
  $$
  \begin{array}{c}
  \displaystyle
    \|\d Y\|_{\cD}
    \leq \cE^{\PP}\left[|\xi|\Ind_{\{|\xi|\geq n\}}+\int_0^T\big|f^0_s\big|\Ind_{\{|f^0_s|\geq n\}}ds\right], 
  \\
  \mbox{and}~ \displaystyle
     \|\d Y\|_{\cS^\b} + \|\d Z\|_{\cH^\b} \nonumber  
     \leq C_{\beta,L,T}\left(\cE^{\PP}\left[|\xi|\Ind_{\{|\xi|\geq n\}}\right]^\beta+\cE^{\PP}\left[\int_0^T|f^0_s|\Ind_{\{|f^0_s|\geq n\}}ds \right]^{\beta}\right). 
 \end{array}
 $$
\end{proposition}

\begin{corollary} \label{cor:estimateYA}
For any  $\d >0$  and  $A\in \cF^{\dbP}_T$ such that $\cE^\dbP[\Ind_A] <\d$ we have
 $$ \sup_{\tau\in\cT_{0,T}} \cE^\dbP [|Y_\tau| \Ind_A] \le \cE^{\dbP}\left[|\xi|\Ind_{\{|\xi|\geq n\}}+\int_0^T\big|f^0_s\big|\Ind_{\{|f^0_s|\geq n\}}ds\right] + C_n \d^\frac12, $$
where $C_n$ is a constant dependent on $n$.
\end{corollary}

\begin{proof}
It is clear that for any $n\in \dbN$
 \begin{align*}
  \sup_{\tau\in\cT_{0,T}} \cE^\dbP [|Y_\tau|\Ind_A] 
   &\leq \sup_{\tau\in\cT_{0,T}} \cE^\dbP [|\d Y_\tau|\Ind_A] +\sup_{\tau\in\cT_{0,T}} \cE^\dbP [|Y^n_\tau|\Ind_A] \\
   &\leq \cE^{\PP}\left[|\xi|\Ind_{\{|\xi|\geq n\}}+\int_0^T\big|f^0_s\big|\Ind_{\{|f^0_s|\geq n\}}ds\right] + 
	\sup_{\tau\in\cT_{0,T}} \cE^\dbP \big[|Y^n_\tau|^2\big]^\frac12  \cE^\dbP [\Ind_A]^\frac12 \\
   &\leq \cE^{\PP}\left[|\xi|\Ind_{\{|\xi|\geq n\}}+\int_0^T\big|f^0_s\big|\Ind_{\{|f^0_s|\geq n\}}ds\right] + C_n \d^\frac12.
 \end{align*}
The second inequality is due to Proposition \ref{prop:deltaYestimate}, and the last inequality is due to the classical estimate on $\LL^2$ solution of BSDE.
\end{proof}

\section{Second-order backward SDE: representation and uniqueness}
\label{sect:2BSDEuniqueness}

 We now prove the following representation theorem for the solution of the 2BSDE \eqref{eq:2BSDE}. 
 Note that this representation implies the uniqueness of the process $Y$, and further that of the process $Z$ as $d\langle Y,B\rangle = Z d\langle B\rangle$.
 
 \begin{theorem} \label{thm:representation}
  Let Assumption \ref{AssumptionOmegaBSDE} hold true and $(Y,Z)$ be a solution to the 2BSDE \eqref{eq:2BSDE} satisfying the minimality condition \eqref{mincond}. 
  For each $\PP\in\cP_0$, let $(\cY^\PP,\cZ^\PP)$ be the solution of the corresponding BSDE \eqref{eq:BSDE}. 
  Then, for any $\PP\in\cP_0$ and $\tau\in\cT_{0,T}$, 
   \begin{equation} \label{eq:2BSDErepresentation}
     Y_\tau = \esssup^\PP_{\PP'\in\cP_+(\tau,\PP)}\cY_\tau^{\PP'}, \quad \PP-a.s.
   \end{equation}
  In particular, the 2BSDE has at most one solution in $\cD\big(\cP_0, \dbF^{+,\cP_0}\big)\times \cH^\beta\big(\cP_0, \dbF^{\cP_0}\big)$ for all $\beta\in(0,1)$ satisfying the minimality condition \eqref{mincond},
    and the comparison result of Proposition \ref{prop:2BSDEcomp} holds true.
 \end{theorem}
 
 \begin{proof}
  The uniqueness of $Y$ is an immediate consequence of \eqref{eq:2BSDErepresentation}, and implies the uniqueness of $Z$, $\widehat{a}_tdt\otimes\cP_0$-q.s.~by the fact that 
          $\langle Y,B\rangle_t = \langle \int_0^\cdot Z_s\cdot B_s,B\rangle_t = \int_0^t\widehat a_s Z_sds$, $\dbP$-a.s. 
  This representation also implies the comparison result as an immediate consequence of the corresponding comparison result of the BSDEs $\cY^{\dbP}$.

\vspace{3mm}
 
This proof of the respresentation is similar to the one in \cite{STZ12}. The only difference is due to the different minimality condition \eqref{mincond}.
  Let $\PP\in\cP_0$ and $\PP'\in\cP_+(\tau,\PP)$ be arbitrary. 
  Since \eqref{2bsdel} holds $\PP'$-a.s., we can consider $Y$ as a supersolution of the BSDE on $\llbracket\tau,T\rrbracket$ under $\PP'$. 
  By comparison result, Proposition \ref{thm:RBSDE.stability.comparison}(ii), we obtain that $Y_\tau\geq\cY_\tau^{\PP'}$, $\PP'$-a.s. 
  As $\cY_\tau^{\PP'}$ (or a $\dbP'$-version of it) is $\cF^+_\tau$-measurable and $Y_\tau$ is $\cF_\tau^{+,\cP_0}$-measurable, we deduce that the inequality also holds $\PP$-a.s., 
     since $\dbP'=\dbP$ on $\cF_\tau^+$ for each $\dbP'\in\cP_+(\tau,\PP)$. 
  Therefore, 
    \begin{equation} \label{eq:2BSDErep1}
      Y_\tau \geq \esssup^\PP_{\PP'\in\cP_+(\tau,\PP)}\cY^{\PP'}_\tau, \quad \PP\mbox{-a.s.}
    \end{equation}
    by arbitrariness of $\PP'$. 
  
  We now show the reverse inequality. 
  Define $\delta Y:=Y-\cY^{\PP'}$ and $\delta Z:=Z-\cZ^{\PP'}$. 
  By Assumption \ref{AssumptionOmegaBSDE}, there exist two bounded processes $a^{\PP'}$ and $b^{\PP'}$ such that 
   \begin{align*}
     \delta Y_\tau &= \int_\tau^T\left(a^{\PP'}_s\delta Y_s + b^{\PP'}_s\cdot\widehat\sigma_s^{\intercal}\delta Z_s\right)ds - \int_\tau^T\widehat\sigma_s^{\intercal}\delta Z_s\cdot dW_s 
                          + \int_\tau^TdK_s^{\PP'} \\
            &= \int_\tau^T a^{\PP'}_s\delta Y_s -  \int_\tau^T\widehat\sigma_s^{\intercal}\delta Z_s\cdot (dW_s - b^{\PP'}_sds) + \int_\tau^TdK_s^{\PP'}, \qquad \PP'\mbox{-a.s.}
   \end{align*}
  Under the measure $\QQ_\tau^{\PP'}$, the process $W_s^{\QQ_\tau^{\PP'}}:=W_s-\int_\tau ^sb^{\PP'}_udu$ is a Brownian motion on $\llbracket\tau, T\rrbracket$ beginning with $W_\tau$.  
  Applying It\^o's formula with $\delta Y_s e^{\int_\tau^s a_u^{\PP'}du}$, 
   $$ \delta Y_\tau = -\int_\tau^Te^{\int_\tau^s a_u^{\PP'}du}\widehat\sigma_s^{\intercal}\delta Z_s\cdot dW_s^{\QQ_\tau^{\PP'}} 
        + \int_\tau^Te^{\int_\tau^s a_u^{\PP'}du}dK^{\PP'}_s, \qquad \PP'\mbox{-a.s.} $$
  Taking conditional expectation with respect to $\QQ_\tau^{\PP'}$ and localization procedure if necessary, we obtain that 
   $$ \delta Y_\tau = \EE^{\QQ_\tau^{\PP'}}\left[\int_\tau^Te^{\int_\tau^s a_u^{\PP'}du}dK^{\PP'}_s \bigg|\cF_\tau^+\right] 
                    \leq e^{L_yT}\EE^{\QQ_\tau^{\PP'}}\big[K^{\PP'}_T-K_\tau^{\PP'}\big|\cF_\tau^+\big]. $$
  By minimality condition \eqref{mincond} 
    $$ 0\leq Y_\tau -\esssup^\PP_{\PP'\in\cP_+(\tau,\PP)}\delta Y_\tau 
       \leq e^{L_yT}\left(\essinf^\PP_{\PP'\in\cP_+(\tau,\PP)}\EE^{\QQ_\tau^{\PP'}}\big[K^{\PP'}_T\big|\cF_\tau^+\big] - K_\tau^\PP\right) = 0, \quad \PP\mbox{-a.s.} $$
  Together with \eqref{eq:2BSDErep1} the assertion follows. 
 \end{proof}

\section{Second-order backward SDE: existence}
\label{sect:2BSDEexistence}

To prove the existence, we first define a value function $V$ by means of the solutions of BSDEs on shifted spaces, then we show that $V$ satisfies the dynamic programming principle, and introduce the corresponding pathwise right limit $V^+$. By combining the standard Doob-Meyer decomposition with our results on reflected BSDEs, we obtain that $V^+$ satisfies the required 2BSDE.

\subsection{Backward SDEs on the shifted spaces}  
  
For every $(t,\omega)\in[0,T]\times\Omega$ and $\PP\in\cP(t,\omega)$, we consider the following BSDE
\begin{equation} \label{bsdeshift}
  \cY_s^{t,\omega,\PP} = \xi^{t,\omega} + \int_s^{T-t}F_r^{t,\omega}(\cY_r^{t,\omega,\PP},\cZ_r^{t,\omega,\PP},\widehat{\sigma}_r)dr - \cZ_r^{t,\omega,\PP}\cdot dB_r, \quad \PP\mbox{-a.e.}
\end{equation}
with $s\in[0,T-t]$. 
By Theorem \ref{thm:ExistenceUniquenessBSDE} we have a unique solution $\big(\cY^{t,\omega,\PP},\cZ^{t,\omega,\PP}\big)\in\cS_{T-t}^\beta(\PP)\times\cH_{T-t}^\beta(\PP)$
 and $\cY^{t,\omega,\PP}\in\cD_{T-t}^\beta(\PP)$.

In this section, we will prove the following measurability result, which is important for the dynamic programming. 
 
\begin{proposition}  \label{prop:BSDE.measurable}
  Under Assumption \ref{assum:BSDE.Lipschitz}, the mapping $(t,\omega,\PP)\mapsto\cY^{t,\omega,\PP}[\xi,T]$ is $\cB([0,T])\otimes\cF_T\otimes\cB(\cM_1)$-measurable.
\end{proposition}

\begin{proof}
  Let $\xi^n$ and $f^n$ be defined as in Section 2. 
  Following Step 1-4 in the proof of \cite[Lemma 4.2]{LRTY20}, we may construct the solution 
            $\left(\cY_s^{n,t,\omega,\PP},\cZ_s^{n,t,\omega,\PP},\cN_s^{n,t,\omega,\PP}\right)$ of the following BSDE
             \begin{align*}
               \cY_s^{n,t,\omega,\PP} = \xi^{n,t,\omega} + \int_s^{T-t}F_r^{n,t,\omega}\big(\cY_r^{n,t,\omega,\PP},\cZ_r^{n,t,\omega,\PP},\widehat{\sigma}_r\big)dr-\int_s^{T-t}\cZ_r^{n,t,\omega,\PP}\cdot dB_r
             \end{align*}
     in a measurable way, such that  
           $(t,\omega,s,\omega',\PP)\mapsto\cY_s^{n,t,\omega,\PP}(\omega')$ is $\cB([0,T])\otimes\cF_T\otimes\cB([0,T])\otimes\cF_T\otimes\cB(\overline\cP_S)$-measurable. 
  By Proposition \ref{prop:deltaYestimate} and $\bf(H2)$ we have that 
       $$ \EE^\PP\left[\sup_{0\leq s\leq T-t}\left|\cY_s^{n,t,\omega,\PP}-\cY_s^{t,\omega,\PP}\right|^\beta\right] \longrightarrow 0, $$
     where $\cY^{t,\omega,\PP}$ is the solution associated with $\big(F^{t,\omega},\xi^{t,\omega}\big)$.
   Then, it follows from \cite[Lemma 3.2]{NN14} that there exists an increasing sequence $\left\{n_k^\PP\right\}_{k\in\NN}\subseteq\NN$ such that  $\PP\mapsto n_k^\PP$ is measurable for each $k\in\NN$ and 
     $$ \lim_{k\to\infty}\sup_{0\leq s\leq T-t}\left|\cY_s^{n_k^\PP,t,\omega,\PP}-\cY_s^{t,\omega,\PP}\right| = 0. $$
   Therefore, $(t,\omega,s,\omega',\PP)\mapsto\cY_s^{t,\omega,\PP}(\omega')$ is $\cB([0,T])\otimes\cF_T\otimes\cB([0,T])\otimes\cF_T\otimes\cB(\overline\cP_S)$-measurable, and the mapping
     $$ (t,\omega,\PP)\mapsto\YY^{t,\omega,\PP}[\xi,T] = \EE^\PP\left[\cY_0^{t,\omega,\PP}\right] $$
     is $\cB([0,T])\otimes\cF_T\otimes\cB(\overline\cP_S)$-measurable. 
   Since $\overline\cP_S\in\cB(\cM_1)$, the mapping $(t,\omega,\PP)\mapsto\YY^{t,\omega,\PP}[\xi,T]$ is $\cB([0,T])\otimes\cF_T\otimes\cB(\cM_1)$-measurable. 
\end{proof}

\begin{lemma} \label{lem:BSDEshift_tower}
 Let Assumptions \ref{assum:BSDE.Lipschitz} and \ref{AssumptionOmegaBSDE} hold true. 
 Then, for all $\tau\in\cT_{0,T}$ and $\PP\in\overline\cP_S$: 
  \begin{enumerate}[(i)]
   \item BSDE \eqref{eq:BSDE} and shifted version \eqref{bsdeshift}:
   $\mathbb{E}^\mathbb{P}\left[\mathcal{Y}^\mathbb{P}_\sigma\Big|\mathcal{F}_\sigma\right](\omega) =\mathbb{Y}^{\sigma, \omega, \mathbb{P}^{\sigma, \omega}}[\xi, T]$, for $\mathbb{P}$-a.e. $\omega\in\Omega$;
   \item Tower property of BSDE: 
           $ Y_{t}[\xi,T] = Y_{t}[Y_\sigma,\sigma] = Y_{t}\Big[\EE\big[Y_\sigma[\xi,T]\big|\cF_\sigma\big],\sigma\Big]. $
  \end{enumerate}
\end{lemma}

We omit the proof as the assertion $(i)$ is a direct result of the uniqueness of the solution to BSDE and the assertion $(ii)$ is similar to \cite[Lemma 2.7]{PTZ18}.

\subsection{Dynamic programming}

We define the value function 
 $$ V_t(\omega):= \sup_{\PP\in\cP(t,\omega)}\YY^{t,\omega,\PP}[\xi,T], \quad\mbox{with}\quad \YY^{t,\omega,\PP}[\xi,T]:=\EE^\PP\left[\cY^{t,\omega,\PP}_0\right]. $$

 Now, we show the dynamic programming result by the measurable selection theorem. We first prove the following class $\cD(\cP)$ integrability result for the process $V$.

\begin{lemma}  \label{lem:ppt-Vsigma}
 Let Assumption \ref{assumption:ppt-prob} and Assumption \ref{AssumptionOmegaBSDE} hold true.  
 Then, the mapping $\omega \mapsto V_\tau(\omega)$ is $\cF^U_\tau$-measurable for each $[0,T]$-valued $\FF$-stopping time $\tau$.
 For any $(t,\omega)\in[0,T]\times\Omega$, 
 \begin{equation*}
   \lim_{n\to\infty}\sup_{\tau\in\cT_{0,T}}\sup_{\PP\in\cP(t,\omega)}\sup_{\QQ\in\cQ_{L_z}(\PP)}\EE^{\QQ}\left[\left|(V_\tau)^{t,\omega}\right|\Ind_{\{|(V_\tau)^{t,\omega}|\geq n\}}\right] = 0.
 \end{equation*}     
\end{lemma}

\begin{proof}
 By the measurability result proved in Proposition \ref{prop:BSDE.measurable} and the measurable selection theorem (see, e.g., \cite[Proposition 7.50]{BS96}), for each $\varepsilon>0$, 
   there exists an $\cF_\tau^U$-measurable kernal $\nu^\varepsilon:\omega\mapsto\nu^{\varepsilon}(\omega)\in\cP\big(\tau(\omega),\omega\big)$, such that for all $\omega\in\Omega$
   \begin{align} \label{eq:selection}
     V_\tau(\omega) \leq \YY^{\tau,\omega,\nu^{\varepsilon}(\omega)}[\xi,T] + \varepsilon. 
   \end{align}   
   This implies that $\omega \mapsto V_\tau(\omega)$ is $\cF^U_\tau$-measurable.
 Further it follows from Lemma \ref{lem:BSDEshift_tower} (i) that 
  \begin{align} \label{eq:consistsol}
     \YY^{\tau, \omega, \nu^\varepsilon(\omega)}[\xi,T] 
       = \EE^{\PP\otimes_\tau\nu^\varepsilon}\left[\mathcal{Y}^{\PP\otimes_\tau \nu^\varepsilon}_\tau\Big|\mathcal{F}_\tau\right](\omega), \quad \mathbb{P}\mbox{-a.s.}\quad 
         \mbox{for each}~~\mathbb{P}\in \mathcal{P}_0.
  \end{align}
 Together with \eqref{eq:selection} we have for $\QQ\in \cQ_{L_z}(\PP)$
  \begin{align*}
    \dbE^\dbQ\big[|V_\tau|\big] 
     & \leq \dbE^\dbQ\left[\EE^{\PP\otimes_\tau\nu^\varepsilon}\left[\big|\mathcal{Y}^{\PP\otimes_\tau \nu^\varepsilon}_\tau\big| \Big|\mathcal{F}_\tau\right]\right] +\e
       \leq \cE^\dbP\Big[\big|\mathcal{Y}^{\PP\otimes_\tau \nu^\varepsilon}_\tau\big| \Big]+ \e \\
     & \leq \cE^\dbP\left[|\xi| + \int_0^T \big|f^0_s\big| ds \right]+\e
       \leq \cE^{\cP_0} \left[|\xi| + \int_0^T \big|f^0_s\big| ds  \right]+\e.
 \end{align*}
 The second last inequality is due to the estimate \eqref{BSDEEstimationEq1} on the BSDE solution. So we have
 \begin{equation} \label{Vuniformbound}
  \sup_{\tau\in\cT_{0,T}} \cE^\cP\big[|V_\tau|\big] \leq \cE^\cP \left[|\xi| + \int_0^T \big|f^0_s\big|ds\right]+\e <\infty.
 \end{equation}
 Further, fix any $\d>0$ and $A\in \cF^U_\tau$ such that $\cE^{\cP_0}[\Ind_A]<\d$. 
 It follows again from \eqref{eq:selection} and \eqref{eq:consistsol} that for $\QQ\in \cQ_{L_z}(\PP)$
  \begin{align*}
    \dbE^\dbQ\big[|V_\tau|\Ind_A\big] 
      & \leq \dbE^\dbQ\left[\EE^{\PP\otimes_\tau\nu^\varepsilon}\left[\big|\mathcal{Y}^{\PP\otimes_\tau \nu^\varepsilon}_\tau\big|\Ind_A \Big|\mathcal{F}_\tau\right]\right] +\e \leq \cE^\dbP\Big[\big|\mathcal{Y}^{\PP\otimes_\tau \nu^\varepsilon}_\tau\big| \Ind_A \Big]+\e \\
      & \leq \cE^\dbP\left[|\xi|\Ind_{\{|\xi|\geq m\}}+\int_0^T\big|f^0_s\big|\Ind_{\{|f^0_s|\geq m\}}ds\right] + C_m \d^\frac12 +\e, \q\mbox{for all $m\in \dbN$}.
  \end{align*}
 The last inequality is due to Corollary \ref{cor:estimateYA}. Now let $m$ be large enough such that
   $$ \cE^{\cP_0}\left[|\xi|\Ind_{\{|\xi|\geq m\}}+\int_0^T\big|f^0_s\big|\Ind_{\{|f^0_s|\geq m\}}ds\right] <\e $$
 and $\d$ be small enough such that $C_m \d^\frac12 <\e$. Then we obtain $ \dbE^\dbQ\big[|V_\tau | \Ind_A\big]  <3\e$. Further note that
 the choice of $m$ and $\d$ is independent from $\dbP$ and $\tau$, so we have
   $$ \sup_{\tau\in\cT_{0,T}}\cE^{\cP_0}\left[ |V_\tau| \Ind_A \right] <3\e. $$
 Finally, since
   $$ \sup_{\tau\in\cT_{0,T}}\cE^{\cP_0}\big[\Ind_{\{|V_\tau \ge n|\}}\big] \le \frac{1}{n}  \sup_{\tau\in\cT_{0,T}}\cE^{\cP_0}[|V_\tau|], $$
  for $n$ big enough, $ \sup_{\tau\in\cT_{0,T}}\cE^{\cP_0}[\Ind_{\{|V_\tau|\ge n\}}]\le \d$
  and thus 
   $ \sup_{\tau\in\cT_{0,T}}\cE^{\cP_0}\big[|V_\tau|\Ind_{\{|V_\tau|\ge n\}}\big] <3\e$.
\end{proof}  
  
  Using the last integrability result, we now show the following results using the same argument as in  \cite{PTZ18} and \cite{LRTY20}.

\begin{proposition}\label{thm:dpp}
Under Assumption \ref{AssumptionOmegaBSDE}, we have 
   \begin{equation}  \label{dppeq}
     V_t(\omega) = \sup_{\PP\in\cP(t,\omega)}\YY^{t,\omega,\PP}\big[V_\tau,\tau\big]
     ~~\mbox{for all}~~
     (t,\omega)\in[0,T]\times\Omega,
     ~\mbox{and}~
      \tau\in\cT_{t,T}.
   \end{equation}
Moreover, we have for all $\PP\in\cP_0$ and $\tau\in\cT_{0,T}$:
   \begin{equation}
     V_t = \esssup_{\PP'\in\cP(t,\PP)}^{\PP}\EE^{\PP'}\left[\cY_t^{\PP'}\big[V_{\tau},\tau\big]\Big|\cF_t\right], \quad \PP\mbox{-a.s.} 
   \end{equation}
\end{proposition}

\begin{proof}
See \cite[Theorem 6.7]{LRTY20}.
\end{proof}

Based on the previous result we can define the right limit of the value function, and the next result shows that $V^+$ is actually a semimartingale under any $\PP\in\cP_0$, and gives its decomposition.

\begin{lemma} \label{lemma:2BSDE-semimart-rep}
 Let Assumptions \ref{assum:BSDE.Lipschitz}, \ref{assumption:ppt-prob} and \ref{AssumptionOmegaBSDE} hold true.  The right limit
   \begin{equation} \label{regularity}
     V^+_t(\omega) :=\lim_{r\in\QQ,r\downarrow t} V_r(\omega)
   \end{equation}
   exists $\cP_0$-q.s.~and the process $V^+$ is c\`adl\`ag $\cP_0$-q.s.
  Also we have:
 \begin{enumerate}[$(i)$]
  \item The process $V^+\in\cD(\cP_0)$.
 
  \item For any $\FF^+$-stopping times $0\leq \tau_1\leq\tau_2\leq T$
        \begin{equation}   \label{V=esssupYV}
          V^+_{\tau_1} = \esssup_{\PP'\in\cP_+(\tau_1,\PP)}^\PP\cY^{\PP'}_{\tau_1}\big[V^+_{\tau_2},\tau_2\big], \quad \PP\mbox{-a.s.}
        \end{equation}
 \end{enumerate}
Further, for any $\PP\in\cP_0$ and $\beta\in(0,1)$, there is $\big(Z^\PP,K^\PP\big)\in\cH^\beta\big(\PP,\FF^{\PP}\big)\times\cI^\beta\big(\PP,\FF^{\PP}\big)$, such that 
  $$ V^+_t = \xi + \int_t^T F_s\big(V^+_s,Z^\PP_s,\widehat\sigma_s\big)ds - \int_t^TZ_s^\PP\cdot dB_s + \int_t^TdK_s^\PP,  \quad \PP\mbox{-a.s.} $$ 
 Moreover, there is some $\FF^{\cP_0}$-progressively measurable process $Z$ which aggregates the family $\big\{Z^\PP\big\}_{\PP\in\cP_0}$. 
\end{lemma}

\begin{proof}
See \cite[Proposition 6.8]{LRTY20} and the step 1 in the proof of the existence part of \cite[Theorem 3.12]{LRTY20}.
\end{proof}

\subsection{Existence through dynamic programming}                                                                                                                                                                                                                                                                                                                                                                                                                                                                                                                                  

Lemma \ref{lemma:2BSDE-semimart-rep} above provides us a candidate $(Y,Z)=(V^+,Z)$ of the solution of the 2BSDE \eqref{eq:2BSDE}. 
Then, it sufficies to verify that the family $\big\{K^\PP\big\}_{\PP\in\cP_0}$ satisfies the minimality condition \eqref{mincond}. 

\begin{proof}
 Let $\PP\in\cP_0$, $\tau\in\cT_{0,T}$ and $\PP'\in\cP_+(\tau,\PP)$ be arbitrary. 
 Let $\big(\cY^{\PP'},\cZ^{\PP'}\big)$ be the solution of 
  $$ \cY_t^{\PP'} = \xi + \int_t^TF_s\big(\cY^{\PP'},\cZ^{\PP'},\widehat\sigma_s\big)ds - \int_t^T\cZ^{\PP'}_s\cdot dB_s, \quad {\PP'}\mbox{-a.s.}  $$
 Define $\delta Y:=V^+-\cY^{\PP'}$ and $\delta Z:=Z-\cZ^{\PP'}$. 
 Then, 
   \begin{align*}
     \delta Y_\tau &= \int_\tau^T\Big(F_s\big(V^+_s,Z_s, \widehat{\sigma}_s\big) - F_s\big(\cY_s^{\PP'},\cZ_s^{\PP'},\widehat{\sigma}_s\big)\Big)ds - \int_\tau^T\delta Z_s\cdot dB_s 
                         + \int_\tau^TdK_s^{\PP'}  \\
          &= \int_\tau^T\big(a_s^{\PP'}\delta Y_s + b^{\PP'}_s\cdot\widehat{\sigma}_s^{\intercal}\delta Z_s\big) ds - \int_\tau^T\widehat{\sigma}_s^{\intercal}\delta Z_s\cdot dW_s 
                         + \int_\tau^TdK_s^{\PP'}, \quad \PP'\mbox{-a.s.} 
   \end{align*}
  where $a^{\PP'}$ and $b^{\PP'}$ are two bounded processes bounded by $L$.  
 Under the measure $\QQ^{\PP'}_{\tau}$, we have 
  $$ \delta Y_\tau = \int_\tau^T a_s^{\PP'}\delta Y_s ds - \int_\tau^T\widehat{\sigma}_s^{\intercal}\delta Z_s\cdot dW^{\QQ_\tau^{\PP'}}_s 
                         + \int_\tau^TdK_s^{\PP'}, \qquad \PP'\mbox{-a.s.} $$
 We next get rid of the linear term in $\delta Y$ by introducing 
   $ \overline{\delta Y}_s:= \delta Y_se^{\int_\tau^s a_u^{\PP'}du}$, $\tau\leq s\leq T $
   so that 
   $$ \delta Y_\tau = -\int_\tau^Te^{\int_\tau^s a_u^{\PP'}du} \widehat{\sigma}_s^{\intercal}\delta Z_s\cdot dW^{\QQ_\tau^{\PP'}}_s 
        + \int_\tau^Te^{\int_\tau^s a_u^{\PP'}du}dK_s^{\PP'}, \quad \PP'\mbox{-a.s.} $$
 Taking conditional expectation with respect to $\QQ_\tau^{\PP'}$ and localization procedure if necessary, we obtain that 
   $$ \delta Y_\tau = \EE^{\QQ_\tau^{\PP'}}\bigg[\int_\tau^Te^{\int_\tau^sa_u^{\PP'}du}dK_s^{\PP'} \bigg| \cF_\tau^+\bigg] 
         \geq e^{-L_yT}\EE^{\QQ_\tau^{\PP'}}\big[K^{\PP'}_T-K^{\PP'}_\tau\big|\cF_\tau^+\big], $$
  therefore, 
   $$ 0\leq \EE^{\QQ_\tau^{\PP'}}\big[K^{\PP'}_T\big|\cF_\tau^+\big]-K^{\PP}_\tau \leq e^{L_yT}\delta Y_\tau. $$
 Then, the result follows immediately thanks to \eqref{V=esssupYV}.
\end{proof}

\appendix 
 
\section{Appendix}

\subsection{Uniform integrability under $\cQ_L(\PP)$}
Here, we show that the space of progressive measurable c\`adl\`ag processes which belong to class $\cD(\PP)$ is complete under the norm 
 $$ \|Y\|_{\cD(\PP)}:=\sup_{\tau\in\cT_{0,T}}\cE^\PP[|Y_\tau|]. $$

First of all, we proof an equivalent characterization of the concept of uniform integrability. 
\begin{proposition}  \label{AppProp1}
  A family $\{X_t\}_{t\in \TT}$ of random variables, where $\TT$ is an arbitrary index set, is uniformly integrable under $\cQ_L(\PP)$, i.e., 
    \begin{equation} \label{AppEq1}
      \lim_{N\to\infty}\sup_{t\in\TT}\cE^{\PP}\big[|X_t|\Ind_{\{|X_t|\geq N\}}\big] = 0, 
    \end{equation}
  if and only if the following two conditions are satisfied:
  \begin{enumerate}[$(a)$]
   \item $\sup_{t\in\TT}\cE^\PP[|X_t|]<\infty$,
   \item For every $\varepsilon>0$, there exists $\delta>0$ such that for any $A\in\cF$ with $\sup_{\QQ\in\cQ_L(\PP)}\QQ[A]<\delta$ we have 
           $$ \sup_{t\in\TT}\cE^{\PP}[|X_t|\Ind_A]<\varepsilon. $$
          
  \end{enumerate}
\end{proposition}
 
\begin{proof}
 Clearly, \eqref{AppEq1} implies $(a)$. 
 Next, let $A\in\cF$ and write $A_t:=\{|X_t|\geq N\}$. Then, we have 
   \begin{align*}
     \sup_{t\in\TT}\cE^\PP[|X_t|\Ind_A] &= \sup_{t\in\TT}\cE^\PP\big[|X_t|\big(\Ind_{A\cap A_t}+\Ind_{A\backslash A_t}\big)\big] \\
       & \leq \sup_{t\in\TT}\cE^\PP\big[|X_t|\Ind_{A_t}\big] + N\sup_{t\in\TT}\cE^\PP[\Ind_{A}] \\ 
       & \leq \sup_{t\in\TT}\cE^\PP\big[|X_t|\Ind_{A_t}\big] + N\sup_{\QQ\in\cQ_L(\PP)}\QQ[A]. 
   \end{align*}
 Given $\varepsilon>0$, by \eqref{AppEq1}, we may find $N$ such that $ \sup_{t\in\TT}\cE^\PP\big[|X_t|\Ind_{A_t}\big]<\frac{\varepsilon}{2}. $
 Therefore, $(b)$ follows by setting $\delta=\frac{\varepsilon}{2N}$. 
 
 Conversely, suppose that $(a)$ and $(b)$ hold. 
 Then, by Markov inequality, we obtain that 
  \begin{align*}
   \sup_{t\in\TT}\sup_{\QQ\in\cQ_L(\PP)}\QQ[|X_t|\geq N] \leq \sup_{t\in\TT}\frac{1}{N}\cE^\PP[|X_t|] \leq \frac{M}{N}, 
  \end{align*}
  where $M$ is the bound indicated in $(a)$. 
 Hence, if $N\geq \frac{M}{\delta}$, then $\sup_{\QQ\in\cQ_L(\PP)}\QQ[A_t]<\delta$, for each $t\in\TT$. 
 By $(b)$, we have for each $t\in\TT$ that
   $ \cE^\PP[|X_t|\Ind_{A_t}]<\varepsilon. $
 Thus, \eqref{AppEq1} follows. 
\end{proof}

Now, we show the completeness of $\cD(\PP)$. 

\begin{theorem}  \label{AppThm1}
 The space $\cD(\PP)$ is complete with respect to the norm $\|\cdot\|_{\cD(\PP)}$.
\end{theorem}

\begin{proof}
  Let $\{X^n\}_{n\in\NN}\subseteq \cD(\PP)$ be a Cauchy sequence with respect to $\|\cdot\|_{\cD(\PP)}$. 
  In particular, this is a Cauchy sequence with respect to $\|\cdot\|_{\DD(\PP)}$. 
  By \cite[VI Theorem 22, Page 83]{DM82}, there exists a c\`adl\`ag process $X$ such that 
  $$ \lim_{n\to\infty}\sup_{t\in[0,T]}(X_t-X_t^n) = 0, \quad \PP\mbox{-a.s.}, $$
  and $\|X\|_{\DD(\PP)}<\infty$.
  Since $\QQ\sim\PP$ for each $\QQ\in\cQ_L(\PP)$, the above convergence holds also for each $\QQ\in\cQ_L(\PP)$. 
  As $\{X^n\}_{n\in\NN}$ is a Cauchy with respect to $\|\cdot\|_{\cD(\PP)}$, for each $\varepsilon>0$ there exists a $N\in\NN$, such that $\|X-X^N\|_{\cD(\PP)}<\varepsilon$, and by triangle inequality we obtain
    $$ \|X\|_{\cD(\PP)} \leq \big\|X-X^N\big\|_{\cD(\PP)} + \big\|X^N\big\|_{\cD(\PP)} < \infty. $$
    
  To show the uniform integrability, it suffices to show $(b)$ in Proposition \ref{AppProp1}. 
  For each $\varepsilon>0$, there exist $N\in\NN$ such that 
     $ \sup_{\tau\in\cT_{0,T}}\cE^\PP\big[\big|X_\tau - X_\tau^N\big|\big]<\frac{\varepsilon}{2}, $
   and $\delta>0$ such that 
     $ \sup_{\tau\in\cT_{0,T}}\cE^\PP\big[\big|X_\tau^N\big|\Ind_A\big]<\frac{\varepsilon}{2}, $
     for each $\sup_{\QQ\in\cQ_L(\PP)}\QQ[A]<\delta$. 
  Therefore, by triangle inequality, we obtain that
    $$ \sup_{\tau\in\cT_{0,T}}\cE^\PP[|X_\tau|\Ind_A] \leq \sup_{\tau\in\cT_{0,T}}\cE^\PP\big[\big|X_\tau - X_\tau^N\big|\big] + \sup_{\tau\in\cT_{0,T}}\cE^\PP\big[\big|X_\tau^N\big|\Ind_A\big]<\varepsilon, $$
    and the assertion follows. 
\end{proof}

\bibliography{BSDE} 

\bibliographystyle{abbrv}

\end{document}